\documentclass[12pt]{amsart} 
\usepackage{amssymb,amsthm,amsmath,a4wide,comment,url}
\usepackage[stable]{footmisc}
\usepackage{amsrefs}
\usepackage{pgf}
\usepackage{ytableau}
\usepackage{pb-diagram}
\usepackage{hyperref}

\begin{document}

\newtheorem{theorem}{Theorem}[section]
\newtheorem{lemma}[theorem]{Lemma}
\newtheorem{definition}[theorem]{Definition}
\newtheorem{proposition}[theorem]{Proposition}
\newtheorem{corollary}[theorem]{Corollary}
\newtheorem{conjecture}[theorem]{Conjecture}
\theoremstyle{remark}
\newtheorem{remark}[theorem]{Remark}
\newtheorem{example}[theorem]{Example}

\newcommand{\dum}{\partial}
\newcommand{\pairs}{\mathfrak{A}}
\newcommand{\soc}{\operatorname{soc}}
\newcommand{\socm}{\mathfrak{soc}}
\newcommand{\smlr}{\ll}
\newcommand{\MW}{\mathcal{MW}}
\newcommand{\MWn}{\dagger}
\newcommand{\nMW}{\circ}
\newcommand{\dom}{\prec}
\newcommand{\supp}{\operatorname{supp}}
\newcommand{\C}{\mathbb{C}}
\newcommand{\Seg}{\operatorname{Seg}}
\newcommand{\Tabs}{\mathcal{T}}
\newcommand{\sh}{\operatorname{sh}}
\newcommand{\GL}{\operatorname{GL}}
\newcommand{\m}{\mathfrak{m}}
\newcommand{\n}{\mathfrak{n}}
\newcommand{\la}{\mathfrak{l}}
\newcommand{\Vien}{\mathcal{K}}
\newcommand{\Vein}{\mathcal{K'}}
\newcommand{\depth}{\mathfrak{d}}
\newcommand{\Mult}{\mathfrak{M}}
\newcommand{\Z}{\mathbb{Z}}
\newcommand{\N}{\mathbb{N}}
\newcommand{\abs}[1]{\left|{#1}\right|}
\newcommand{\nextind}[1]{#1^+}
\newcommand{\prevind}[1]{#1^-}
\renewcommand{\subset}{\subseteq}
\renewcommand{\supset}{\supseteq}
\newcommand{\RSK}{\mathcal{RSK}}
\newcommand{\RSKn}{\operatorname{RSK}'}
\newcommand{\id}{\operatorname{Id}}
\newcommand{\Lads}{\mathcal{L}}
\newcommand{\Lad}{\operatorname{Lad}}
\newcommand{\lshft}[1]{\overset{\leftarrow}{#1}\vphantom{#1}}               
\newcommand{\spcl}{\flat}
\newcommand{\ldr}[1]{#1^{\natural}}
\newcommand{\ldrr}[1]{\la(#1)}
\newcommand{\drv}[1]{#1'}
\newcommand{\inv}[1]{#1^{\diamond}}
\newcommand{\std}{\Lambda}

\newcommand{\anti}{\searrow}
\newcommand{\bA}{\mathbb{A}}
\newcommand{\bP}{\mathbb{P}}
\newcommand{\bQ}{\mathbb{Q}}
\newcommand{\bY}{\mathbb{Y}}
\newcommand{\cH}{\mathcal{H}}
\newcommand{\cV}{\mathcal{V}}
\newcommand{\Col}{\mathrm{Col}}
\newcommand{\fm}{\mathfrak{m}}
\newcommand{\kh}{\hat{k}}
\newcommand{\Row}{\mathrm{Row}}
\newcommand{\shape}{\mathrm{shape}}
\newcommand{\Tab}{\mathrm{Tab}}
\newcommand{\ub}{u^\bullet}
\newcommand{\ubd}{u^{\bullet\dagger}}
\newcommand{\Uh}{\hat{U}}
\newcommand{\vb}{v^\bullet}
\newcommand{\word}{\mathrm{word}}
\newcommand{\wt}{\mathrm{wt}}
\newcommand{\invert}{\mathrm{inv}}
\newcommand{\flip}{\nearrow\!\!\!\!\!\!\swarrow}
\newcommand{\ITab}{\mathrm{ITab}}
\newcommand{\cRSK}{\mathrm{cRSK}}

\title[RSK for $GL_n$ over local fields]{Robinson-Schensted-Knuth correspondence in the representation theory
of the general linear group over a non-archimedean local field\\
(with an appendix by Mark Shimozono)}
\author{Maxim Gurevich}
\address{Department of Mathematics, Technion -- Israel Institute of Technology, Haifa, Israel.}
\email{maxg@technion.ac.il}
\author{Erez Lapid}
\address{Department of Mathematics, Weizmann Institute of Science, Rehovot, Israel}
\email{erez.m.lapid@gmail.com}
\date{\today}

\begin{abstract}
We construct new ``standard modules'' for the representations
of general linear groups over a local non-archimedean field.
The construction uses a modified Robinson-Schensted-Knuth correspondence
for Zelevinsky's multisegments.

Typically, the new class categorifies the basis of Doubilet, Rota, and Stein for matrix polynomial rings, indexed by bitableaux.
Hence, our main result provides a link between the dual canonical basis (coming from quantum groups) and the DRS basis.
\end{abstract}

\maketitle

\setcounter{tocdepth}{1}
\tableofcontents

\section{Introduction}
The Zelevinsky classification (\cite{MR584084}) is one of the cornerstones of the representation theory of reductive groups over a non-Archimedean local field $F$.
It classifies the equivalence classes of the irreducible (complex, smooth) representations of the general linear groups of all ranks over $F$
in terms of multisegments, which are essentially a combinatorial object.

The irreducible representations may all be obtained as socles of the so-called standard modules (that are also indexed by multisegments).
The standard modules are the parabolic induction of certain irreducible, essentially square-integrable representations,
and constitute a basis for the Grothendieck group.
The situation is analogous to that in category $\mathcal{O}$ where the Verma modules play the role of standard modules.
This is in fact not a coincidence. It is explained by the Arakawa--Suzuki functors \cite{MR1652134}
which provide a link between category $\mathcal{O}$ of type $A$ and representations of $\GL_n(F)$, $n\ge0$.
See \cites{MR2320806, MR3495794, MR3866895} for more details.

In this paper we present a new class of \textit{RSK-standard} modules, that are parabolically induced from ladder representations.
Its construction relies on an application of the well-known Robinson--Schensted--Knuth (RSK) correspondence.
The new class is again in bijection with irreducible representations, in such a way that each irreducible representation is realized as a subrepresentation
(and conjecturally, the socle) of the corresponding RSK-standard module.

Ladder representations are a class of irreducible representations with particularly nice properties \cites{MR3163355, MR2996769, MR3495794, 10.1093/imrn/rnz006}.
In particular, parabolic induction from two ladder representations is well understood.

Let us describe the new construction in more detail. Roughly speaking, a multisegment $\m$ is a collection of pairs of integers $[a_i,b_i]$, $i=1,\dots,n$, with $a_i\le b_i$.
Denote by $Z(\m)$ the irreducible representation of $GL_N(F)$ corresponding to $\m$, as defined by Zelevinsky.

Now, the RSK correspondence attaches to $\m$ a pair of semistandard Young tableaux of the same shape
of total size $n$.
For our purposes it will be more convenient to use a modified version of RSK, $\m\mapsto(P_{\m},Q_{\m})$
where $P_{\m}$ and $Q_{\m}$ are \emph{inverted} Young tableaux (of the same shape, of total size $n$).
By an inverted Young tableau we mean that the entries in each row are strictly decreasing and
the entries in each column are weakly decreasing,
unlike the usual convention 
in which the entries in each row are weakly increasing and the entries in each column are strictly increasing.
As in the classical case, $P_{\m}$ is filled by the $a_i$'s and $Q_{\m}$ by the $b_i$'s.

Suppose that the pair $(P_{\m}, Q_{\m})$ is given by
\ytableausetup{mathmode,boxsize=2em}
\[
P_{\m}=\begin{ytableau}
c_{1,1} & c_{1,2} & \dots & c_{1,\lambda_2}
& \dots & c_{1,\lambda_1} \\
c_{2,1}    & c_{2,2} & \dots
& c_{2,\lambda_2} \\
\vdots & \vdots
& \vdots \\
c_{k,1} & \dots & c_{k,\lambda_k}
\end{ytableau}
\ \ \
Q_{\m}=\begin{ytableau}
d_{1,1} & d_{1,2} & \dots & d_{1,\lambda_2}
& \dots & d_{1,\lambda_1} \\
d_{2,1}    & d_{2,2} & \dots
& d_{2,\lambda_2} \\
\vdots & \vdots
& \vdots \\
d_{k,1} & \dots & d_{k,\lambda_k}
\end{ytableau}
\]
To each row of the resulting shape we attach a ladder representation by setting
\[
\pi_i = Z( [c_{i,1},d_{i,1}] +\ldots + [c_{i,\lambda_i},d_{i,\lambda_i}] ),\quad i=1,\ldots,k\;.
\]
In other words, $\pi_i$ is the irreducible representation whose defining multisegment is encoded in the pair of $i$-th rows of the tableaux $(P_{\m}, Q_{\m})$.

The RSK-standard module attached to $\m$ is now defined as
\[
\std(\m) = \pi_k\times\cdots \times \pi_1\;,
\]
where the (Bernstein--Zelevinsky) product denotes (normalized) parabolic induction.

An intriguing angle on the new class of modules comes from combinatorial invariant theory.
Let $\mathcal{R}$ be the direct sum of Grothendieck groups of smooth finite-length representations, of all $GL_n(F),\,n\geq0$.
The Bernstein--Zelevinsky product makes $\mathcal{R}$ into a commutative ring and Zelevinsky showed that $\mathcal{R}$ is the polynomial ring over $\Z$ freely generated by the classes
of irreducible, essentially square-integrable representations.
In other words, the standard modules form the monomial basis for $\mathcal{R}$ (as a $\Z$-module).
On the other hand, in the context of invariant theory, Doubilet, Rota, and Stein constructed a basis for polynomial matrix rings,
parameterized by pairs of tableaux, consisting of products of minors \cites{MR0498650} (see also \cites{MR0485944, 1605.06696}.
As will be explained in \S\ref{sec: rota}, the RSK-standard modules \textit{categorify} the DRS basis for matrix polynomial rings.
This phenomenon follows from the paramount determinantal property \cite{MR3163355} of ladder representations.
The construction of $\std(\m)$ as a product of ladders parallels the recipe of DRS for obtaining all basis elements as products of minors.

Our main result is the following.
\begin{theorem}\label{thm: intro}
For every multisegment $\m$, $Z(\m)$ occurs as a subrepresentation of $\std(\m)$. More precisely,
\[
Z(\m) \cong \soc(\soc(\dots\soc(\soc(\pi_k\times\pi_{k-1})\times\pi_{k-2})\times \dots\, \times\pi_2)\times \pi_1)\;.
\]
\end{theorem}
Here, $\soc(\tau)$ stands for the socle (i.e., the maximal semisimple subrepresentation) of a representation $\tau$.

We expect (Conjecture \ref{conj: main}) that as in Zelevinsky's case, $\soc(\std(\m))$ is itself irreducible (hence isomorphic to $Z(\m)$)
and occurs with multiplicity one in the Jordan--H\"older sequence of $\std(\m)$.

Note that the parameter $k= k(\m)$, i.e. the number of rows in the tableaux $P_{\m},Q_{\m}$, is the \textit{width} of the multisegment, as defined and studied in \cite{10.1093/imrn/rnz006}.
In particular, it was shown in [ibid.] that $k$ is the minimal number of ladder representations whose product contains $Z(\m)$ as a subquotient.
In that respect, the RSK-standard modules possess a minimality property.

The case of Theorem \ref{thm: intro} with $k(\m)=2$ and certain regularity conditions was previously shown in
\cite{1711.01721} using quantum shuffle methods and equivalences to module categories of quiver Hecke algebras.
In fact, the papers \cites{10.1093/imrn/rnz006, 1711.01721} were the point of departure of the present work.

Coming back to the categorification context, irreducible representations are known to give the (dual) \textit{canonical basis}
(\`{a} la Lusztig or Kashiwara) of $\mathcal{R}$ (see \cite{MR1985725}, for example).
Thus, Theorem \ref{thm: intro} gives a certain link between the DRS basis and the canonical basis.
Similar links were explored in \cite{MR1392496} by means of quantum matrix rings.
Our results and expectations open a new perspective on those links in a categorial setting, i.e. it is the submodule structure which is used to characterize relations between those two bases.

A natural question which arises is what can be said about the other irreducible constituents of $\std(\m)$.
Based on empirical evidence, we conjecture that for any irreducible subquotient $Z(\n)$ of $\std(\m)$ other than $Z(\m)$,
the RSK-data $(P_{\n},Q_{\n})$ is strictly smaller than $(P_{\m},Q_{\m})$ with respect to the product order of the domination order of inverted Young tableaux (see \S\ref{sect: tri}).
Once again, this would be analogous to the situation for Zelevinsky classification, where the pertinent partial order on multisegments (as originally defined in
\cite{MR584084}) is closely related to the Bruhat order on the symmetric group.

In particular, this would give a positive answer to the question which still remains open, of whether the classes of RSK-standard modules indeed form a $\Z$-basis for
(an appropriate subgroup of) $\mathcal{R}$. Certain related questions can be answered (Proposition \ref{prop: basis-rota}) using the basis theorem of \cite{MR0485944}.

Recall that the classical RSK correspondence admits several (equivalent) implementations.
For our purposes, it is best to use Knuth's algorithm \cite{MR0272654}*{\S4} which constructs the Young tableaux row by row,
rather than the earlier (and perhaps more commonly used) Robinson-Schensted insertion/bumping algorithm which fills them box by box.
(A pictorial approach to Knuth's algorithm was given by Viennot \cite{MR0470059}.)

A key part of the proof of Theorem \ref{thm: intro} is an intriguing relation between the Knuth algorithm and the
description of the socles of certain induced representations due to the second-named author and M\'inguez \cite{MR3573961}.
In turn, this description is closely related to the  M\oe glin--Waldspurger algorithm
for the Zelevinsky involution \cite{MR863522}*{\S II.2}.
In fact, the main new combinatorial input (Corollary \ref{cor: main}), which is interesting in its own right,
is that roughly speaking, under certain conditions the Knuth algorithm commutes with the first step of the M\oe glin--Waldspurger algorithm.

In an appendix by Mark Shimozono, further illuminating relations between those two combinatorial algorithms are presented.
The slightly different point of view of the appendix suggests that some of our questions may be resolved in future work through the combinatorial theory of tableaux and crystals.

The connection between RSK and the M\oe glin--Waldspurger algorithm is further elucidated by considering \textit{enhanced multisegments}, that is, ``multisegments with dummies''.
More precisely, starting with a (genuine) multisegment $\m=\Delta_1+\dots+\Delta_k$ we add empty (dummy) segments of the form $[a,a-1]$, $a\in\Z$.
This will not affect $Z(\m)$ but will change the output of RSK. In Theorem \ref{thm: dummy}, to any irreducible representation $\pi=Z(\m)$ we
concoct a whole slew of ``standard modules'' admitting $\pi$ as a subrepresentation (and conjecturally, as the socle)
by adding dummy segments to $\m$ at will.

Let $n_a$, $a\in\Z$ be the number of dummy segments of the form $[a,a-1]$.
There are two extreme cases, as stated in Proposition \ref{prop: std}. If $n_a\ll n_{a+1}$ for all $a$ in the support of $\m$ then the standard module
will be the one defined by Zelevinsky (taking the product of $Z(\Delta_i)$ in a suitable order).
On the other hand if $n_a\gg n_{a+1}$ for all $a$ in the support of $\m$, then we will get the standard module in the sense of Langlands.
In other words, the output of RSK essentially coincides (up to dummy segments) with the M\oe glin--Waldspurger algorithm.
Thus, we get a family of standard modules interpolating between the standard modules of Zelevinsky and Langlands.

Both the Zelevinsky classification and the RSK correspondence admit geometric interpretations
(starting with \cites{MR617466, MR783619} for the former, \cites{MR672610, MR929778, MR2979579} for the latter).
However, we are not aware of a geometric interpretation of the abovementioned partial order defined through RSK.
It would be interesting to find a geometric interpretation of the results and conjectures presented here.

Finally, it would be interesting to know whether other type $A$ representation categories, such as those of quiver Hecke algebras (through the Brundan-Kleshchev isomorphisms \cite{MR2551762}),
quantum affine algebras (through quantum affine Schur-Weyl duality \cite{MR1405590}) or even the classical category $\mathcal{O}$
(through the Arakawa-Suzuki functors \cites{MR1652134, MR2320806}), can be used to reinterpret the results of this paper.

Part of this work was carried out while the authors participated in the month-long activity
``On the Langlands Program: Endoscopy and Beyond'' at
the Institute for Mathematical Sciences of the National University of Singapore,
Dec.~2018 -- Jan.~2019. The first-named author and Mark Shimozono took part in the thematic trimester program on representation theory in the Institut Henri Poincar\'{e}, Paris, Jan.~2020,
where the fruitful conditions of the venue lead to the addition of Appendix A. We would like to thank the IMS, the IHP and the organizers of the two programs
for their warm hospitality.

\section{Operations on multisegments}
\subsection{Multisegments}
A \emph{segment} $\Delta$ (of length $b-a+1$) is a subset of $\Z$ of the form
\[
[a,b]=\{n\in\Z:a\le n\le b\}
\]
for some integers $a\le b$.
We will write $b(\Delta)=a$ and $e(\Delta)=b$.
We also write $\lshft\Delta=[a-1,b-1]$ 
and ${}^-\Delta=[a+1,b]$. Note that if $a=b$, then ${}^-\Delta$ is the empty set.

We denote by $\Seg$ the set of all segments.

Given $\Delta_1$, $\Delta_2\in\Seg$, we write $\Delta_1\smlr\Delta_2$ if $b(\Delta_1)<b(\Delta_2)$
and $e(\Delta_1)<e(\Delta_2)$. This is a strict partial order on $\Seg$.

A \emph{multisegment} is a multiset of segments, i.e., a formal finite sum
$\m=\sum_{i\in I}\Delta_i$ where $\{\Delta_i\in \Seg\}_{i\in I}$ are segments.

If $\m\ne0$, we will write $\min\m=\min_{i\in I}b(\Delta_i)$.

The size $|\m|$ of a multisegment $\m$ is by definition, the sum $\sum_{i\in I} |\Delta_i|$ of the cardinalities of its segments.

For any set $X$, we denote by $\N(X)$ the ordered commutative monoid of finite multisets on $X$,
that is, finitely supported functions from $X$ to $\N=\Z_{\ge0}$. We view elements of $\N(X)$ as formal finite (possibly empty) sums
of elements of $X$.

From this point of view, we denote by $\Mult:= \N(\Seg)$ the ordered monoid of multisegments. If $\m'\le\m$ we say that $\m'$ is a sub-multisegment of $\m$.

A \emph{ladder} is a nonzero multisegment of the form $\sum_{i=1}^k\Delta_i$, for segments $\Delta_i$,
where $\Delta_{i+1}\smlr\Delta_i$ for $i=1,\dots,k-1$.

We will write $\Lad\subset \Mult$ for the collection of ladders.

\subsection{M\oe glin--Waldspurger involution} \label{sec: MW}
Let us recall the M\oe glin--Waldspurger algorithm from \cite{MR863522}*{\S II.2}.
More precisely, we will adopt a mirror variation on [loc. cit.], in which minimal segments are taken in place of maximal ones.\footnote{This is justified by the fact that the M\oe glin--Waldspurger algorithm is compatible with contragredient.}

For any multisegment $0\neq \m = \sum_{i\in I} \Delta_i \in \Mult$ we define a segment $\Delta^{\nMW}(\m)$
and a multisegment $\m^{\MWn}$ as follows.

Let $i_1\in I$ be such that $b(\Delta_{i_1})=\min\m$ and $e(\Delta_{i_1})$ is minimal.
If there is no $i\in I$ such that $\Delta_{i_1}\smlr\Delta_i$ and $b(\Delta_i)=b(\Delta_{i_1})+1$,
then we set $\Delta^{\nMW}(\m):=[\min\m,\min\m]\in \Seg$.
Otherwise, let $i_2\in I$ be such that $\Delta_{i_1}\smlr\Delta_{i_2}$, $b(\Delta_{i_2})=b(\Delta_{i_1})+1$ and $e(\Delta_{i_2})$ is minimal with respect to these properties.
Continuing this way, we define an integer $k>0$ and indices $i_1,\dots,i_k\in I$, such that
\begin{enumerate}
\item For all $j<k$, $\Delta_{i_j}\smlr\Delta_{i_{j+1}}$,
$b(\Delta_{i_{j+1}})=b(\Delta_{i_j})+1$, and $e(\Delta_{i_{j+1}})$ is minimal with respect to these properties.
\item There does not exist $i\in I$ such that $\Delta_{i_k}\smlr\Delta_i$ and $b(\Delta_i)=b(\Delta_{i_k})+1$.
\end{enumerate}
We set $\Delta^{\nMW}(\m):=[b(\Delta_{i_1}),b(\Delta_{i_k})]=[\min\m,\min\m+k-1]\in \Seg$.

We call $I^*=\{i_1,\dots,i_k\}\subset I$ a set of \emph{leading indices} of $\m$. \label{sec: leading}
It is not unique, but the set $\{\Delta_{i_1},\dots,\Delta_{i_k}\}$ depends only on $\m$.
We also write $i^*_{\min}=i_1$, $i^*_{\max}=i_k$ and denote by
$i\mapsto\nextind{i}:I^*\setminus\{i^*_{\max}\}\rightarrow I^*\setminus\{i^*_{\min}\}$
the bijection $i_j\mapsto i_{j+1}$. The inverse will be denoted by $i\mapsto\prevind{i}$.

We set $\m^{\MWn}:=\sum_{i\in I}\Delta_i^*\in \Mult$ (discarding the summands that are empty sets), where
\begin{equation} \label{def: *}
\Delta_i^*=\begin{cases}{}^-\Delta_i&{i\in I^*,}\\\Delta_i&\text{otherwise.}\end{cases}
\end{equation}

We denote the resulting map $\m\mapsto(\m^{\MWn},\Delta^{\nMW}(\m))$ by
\[
\MW : \Mult\setminus\{0\} \to \Mult \times \Seg\;.
\]
We say that $\m$ is \textit{non-degenerate} if $\Delta_i^*\ne\emptyset$, for all $i\in I$.
Equivalently, $\Delta_i\ne [\min\m,\min\m]$ for all $i\in I$. \label{sec: ndgnrt}

Applying the map $\MW$ repeatedly we obtain the M\oe glin--Waldspurger involution\footnote{
The fact that $\m\mapsto\m^\#$ is an involution is not obvious from the definition.} $\m\mapsto\m^\#$ on $\Mult$
which is the combinatorial counterpart of the Zelevinsky involution \cite{MR584084}*{\S9}.
More precisely, $\m^\#$ is defined recursively by
\[
0^\#=0,\ \ \m^\#=\Delta^{\nMW}(\m)+(\m^{\MWn})^\#.
\]
In particular, the map $\MW$ is injective.

The subset $\Lad$ is preserved under $\m\mapsto\m^\#$.

\subsection{Modified RSK correspondence for multisegments} \label{sec: Vien}


Let $n\ge0$ be an integer and $\lambda=(\lambda_1,\dots,\lambda_k)$ a partition of $n$,
i.e. $\lambda_1\ge\dots\geq \lambda_k>0$ are integers and $\lambda_1+\dots+\lambda_k=n$.
The partition $\lambda$ gives rise to a Young diagram of size $n$.
The conjugate partition $\lambda^t:=\mu=(\mu_1,\dots,\mu_l)$ is given by $l=\lambda_1$, $\mu_j=\max\{i:\lambda_i\ge j\}$.

A semistandard Young tableau of shape $\lambda$ is a filling of a Young diagram of shape $\lambda$
by integers, such that the entries along each row are weakly increasing and the entries down each column are strictly increasing.

The classical Robinson--Schensted--Knuth (RSK) correspondence is a bijection between
$\N(\Z\times\Z)$ (i.e., multisets of pairs of integers) and
pairs of semistandard Young tableaux of the same shape.
We refer to \cite{MR1464693}*{\S4} or \cite{MR1676282}*{\S7.11-13} for standard references on RSK.

We will consider a slight modification of RSK where semistandard Young tableaux are replaced
by \textit{inverted Young tableaux}.
By definition, an inverted Young tableau of shape $\lambda$ is a filling of
the Young diagram of $\lambda$ by integers, i.e.
a double sequence $z_{i,j}\in\Z$, $i=1,\dots,k$, $j=1,\dots,\lambda_i$,
that satisfies
\[
z_{i,1}>\dots> z_{i,\lambda_i},\quad\forall i=1,\dots,k\;,
\]
\[
z_{1,j}\ge\dots\ge z_{\mu_j,j},\quad \forall j=1,\dots,l\;.
\]
We denote by $\mathcal{T}$ the set of pairs of inverted Young tableaux $(P,Q)$ of the same shape.
(See \cite{MR2979579} for a similar nonstandard convention.
The appendix of \cite{MR1464693} also discusses closely related variants of RSK.)

Thus, the modified RSK correspondence is a bijection
\[
\RSKn: \N(\Z\times\Z) \longrightarrow \mathcal{T}.
\]
It can be defined using a modification of the Schensted insertion/bumping algorithm where we replace
strict inequalities by weak inequalities in the opposite direction and vice versa.
It is advantageous, however, to use a modification of the Knuth algorithm which we will recall below.
We remark that if $\RSKn( \sum_{i\in I} (a_i,b_i)) = (P,Q)$, the $a_i$'s and the $b_i$'s comprise
the entries of the tableaux $P$ and $Q$, respectively.

We may identify $\Seg$ as a subset of $\Z\times \Z$ by $\Delta \mapsto (b(\Delta),e(\Delta))$.
Hence, we may identify $\Mult$ with a subset of $\N(\Z\times\Z)$.
Thus, for any multisegment $\m\in \Mult$ we may consider the pair of inverted Young tableaux
\[
(P_{\m},Q_{\m}): = \RSKn(\m)\in \mathcal{T}.
\]
In what follows we will only consider the restriction of $\RSKn$ to $\Mult$.

\subsubsection{Ladders and tableaux}\label{sect: lad}

First, we would like to be able to describe certain elements of $\mathcal{T}$ in terms of ladders.

Let $\la_1=\sum_{i=1}^k\Delta_i$ and $\la_2=\sum_{i=1}^{k'}\Delta'_i$ be two ladders with
$\Delta_{i+1}\smlr\Delta_i$, $i=1,\dots,k-1$ and $\Delta'_{i+1}\smlr\Delta'_i$, $i=1,\dots,k'-1$.

We say that $\la_2$ is \emph{dominant} with respect to $\la_1$, if $k'\ge k$ and $\lshft\Delta_i\smlr\Delta'_i$ for all
$i=1,\dots,k$.

We say that the pair $(\la_2,\la_1)$ is \emph{permissible} if $\la_2$ is dominant with respect to $\la_1$ and
for all $i=1,\dots,k$ and $j=1,\dots,k'$
such that $\lshft\Delta_i\smlr\Delta'_j$ and either $j=k'$ or $\lshft\Delta_i\not\smlr\Delta'_{j+1}$
(and in particular, $j\ge i$), we have $e(\Delta_r)\ge b(\Delta'_{j-i+r})$ for all $r=1,\dots,i$.

Let us write $\Lads'$ for the collection of tuples $(\la_1,\dots,\la_m)\in \Lad^m$, for some $m$, such that $\la_i$ is dominant with respect to $\la_{i+1}$, for all $1\leq i<m$.

We may think of $\Lads'$ as a subset of $\mathcal{T}$ as follows.
Given $\underline{\la} = (\la_1,\dots,\la_m)\in \Lads'$
where $\la_i = \sum_{j=1}^{n_i} \Delta^i_j$ with $\Delta^i_{j+1} \smlr \Delta^i_j$ we construct
\[
(P(\underline{\la}), Q(\underline{\la}))\in \mathcal{T}
\]
by letting the $(i,j)$-th entry in $P(\underline{\la})$ (resp. $Q(\underline{\la})$)
be $b(\Delta^i_j)$ (resp. $e(\Delta^i_j)$).
The map $\underline{\la}\mapsto (P(\underline{\la}), Q(\underline{\la}))$ is an injection of
$\Lads'$ into $\mathcal{T}$. Its image is the set of pairs $(P,Q)$ of tableaux of the same shape such that
$P_{i,j}\leq Q_{i,j}$ for all entries of the tableaux.

Finally, we denote by $\Lads\subset \Lads'$ the subset consisting of tuples $(\la_1,\dots,\la_m)\in \Lads'$ such that $(\la_i,\la_j)$
is permissible for all $1\le i<j\le m$.

\subsubsection{The Knuth implementation}\label{subsection: implement}
Let us fix a multisegment $0\neq \m=\sum_{i\in I}\Delta_i\in \Mult$.
We will explicate $(P_{\m},Q_{\m})$ introducing some terminology for multisegments.

We define the \emph{depth function}
$\depth=\depth_{\m}:I\rightarrow\Z_{\ge0}$ by
\begin{align*}
\depth(i)=\max\{j:\exists i_0=i,i_1,\dots,i_j\in I\text{ such that }&\Delta_{i_{r}}\smlr\Delta_{i_{r+1}}
,r=0,\dots,j-1\}\;.
\end{align*}
We let $d=d(\m)=\max_{i\in I}\depth(i)$ be the depth of $\m$.

It is clear that if $\depth(i)=\depth(j)$, then
either $\Delta_i=\Delta_j$ or $\Delta_i\supsetneq\Delta_j$ or $\Delta_i\subsetneq\Delta_j$.

\begin{lemma} \ \label{lem: ib}
\begin{enumerate}
\item \label{itm: ib-first}
For any $i\in I$ and $0\le k<\depth(i)$ there exists $i'\in I$ such that $\Delta_i\smlr\Delta_{i'}$ and $\depth(i')=k$.
\item \label{itm: ib-second}
Let $i_1,i_2\in I$ be such that $\depth(i_1)=\depth(i_2)$ and $\Delta_{i_2}\subset\Delta_{i_1}$.
Then for any $i\in I$ such that 
$\Delta_{i_2}\subset\Delta_i\subset\Delta_{i_1}$
and $\depth(i)>\depth(i_1)$, there exists $j\in I$ such that $\Delta_{i_2}\subset\Delta_j\subset\Delta_{i_1}$,
$\depth(j)=\depth(i_1)$ and $\Delta_i\smlr\Delta_j$.
\end{enumerate}
\end{lemma}

\begin{proof}
The first part is clear.
To prove the second part we argue by induction on $\depth(i)-\depth(i_1)$.
Let $i'$ be such that $\Delta_i\smlr\Delta_{i'}$ and $\depth(i')=\depth(i)-1$.
Then, $b(\Delta_{i_1})\le b(\Delta_i)<b(\Delta_{i'})$.
However, we cannot have $\Delta_{i_1}\smlr\Delta_{i'}$ since otherwise $\depth(i_1)\ge\depth(i')+1=\depth(i)$.
Therefore $\Delta_{i'}\subset\Delta_{i_1}$.
Similarly, we cannot have $\Delta_{i_2}\smlr\Delta_{i'}$ and therefore $\Delta_{i'}\supset\Delta_{i_2}$.
If $\depth(i')=\depth(i_1)$ we are done. Otherwise, we apply the induction hypothesis to $i'$.
\end{proof}

For any $k=0,\dots,d$, choose an \emph{admissible enumeration} $\{i^k_1,\dots,i^k_l\}$ of $\depth^{-1}(k)$, namely such that
$\Delta_{i^k_1}\supset\dots\supset\Delta_{i^k_l}$, and set $j_k:=i^k_l$.

We will say that $i^k_r$ is a \emph{distinguished index} (with respect to the enumeration) if either $r=l$ or $\Delta_{i^k_{r+1}}\ne\Delta_{i^k_r}$.
\label{sec: distin}

Let $\sigma$ be the permutation of the index set $I$, whose cycle decomposition is given by $\{(i^k_1,\dots,i^k_l)\}_{k=0,\dots,d}$.

Define
\begin{subequations} \label{def: Vien}
\begin{gather}
\Delta_i'=[b(\Delta_i),e(\Delta_{\sigma(i)})]\in \Seg,\ \ i\in I,\\
\ldr{I}=\{j_0,\dots,j_d\},\ I'=I\setminus\ldr{I}\\
\ldrr{\m}=\sum_{i\in\ldr{I}}\Delta'_i,\ \drv\m=\sum_{i\in I'}\Delta'_i\in \Mult\;.
\end{gather}
\end{subequations}
For reference, we will write $i^\vee = \sigma(i)$, for all $i\in I$.

Note that for any $i\in\ldr{I}$ we have
\begin{equation}\label{eq: max}
b(\Delta'_i)=\max_{j\in I:\depth(j)=\depth(i)}b(\Delta_j),\quad e(\Delta'_i)=\max_{j\in I:\depth(j)=\depth(i)}e(\Delta_j)\;.
\end{equation}
Thus, it follows from Lemma \ref{lem: ib}\eqref{itm: ib-first} that $\ldrr{\m}$ is a ladder.

\begin{lemma}\label{lem: id}
Suppose that $i,i'\in I$ are such that $\Delta_{i^\vee}\subsetneq \Delta_{i'}\subsetneq \Delta_i$.
Then, $\depth_{\m}(i')< \depth_{\m}(i)$.
\end{lemma}

\begin{proof}
Assume on the contrary that $\depth_{\m}(i')\ge \depth_{\m}(i)$.
By Lemma \ref{lem: ib}\eqref{itm: ib-second}, there is $j\in I$, for which $\depth_{\m}(j) = \depth_{\m}(i)$, $\Delta_{i^\vee}\subseteq \Delta_{j}\subseteq \Delta_i$,
and either $i'=j$ or $\Delta_{i'}\smlr \Delta_j$. In both cases, it is clear that $\Delta_{i^\vee}\subsetneq \Delta_{j}\subsetneq \Delta_i$. This contradicts the definition of $i^\vee$.
\end{proof}

\begin{lemma}\label{lem: ic}
For all $i\in I'$ we have $\depth_{\m'}(i)\le\depth_{\m}(i)$.
\end{lemma}

\begin{proof}
Note that when $i,j\in I'$ and $\depth_{\m}(i)=\depth_{\m}(j)$,
either $\Delta'_i=\Delta'_j$ or $\Delta'_i\subsetneq\Delta'_j$ or $\Delta'_i\supsetneq\Delta'_j$.

Suppose that $\Delta'_{j_k} \smlr \ldots \smlr \Delta'_{j_1} \smlr \Delta'_{j_0}$ for given indices $j_0,\ldots,j_k\in I'$.
We need to show that $k\leq \depth_{\m}(j_k)$.

By the remark above $\depth_{\m}(j_{c+1})\ne\depth_{\m}(j_c)$ for all $c<k$.
Suppose on the contrary that $\depth_{\m}(j_{c+1})<\depth_{\m}(j_c)$ for some $c$.
Since $b(\Delta_{j_c})=b(\Delta'_{j_c})>b(\Delta'_{j_{c+1}})=b(\Delta_{j_{c+1}})$
we necessarily have $\Delta_{j_c}\subsetneq \Delta_{j_{c+1}}$.
Also, by definition of $\m'$, there exists $i_1\in I$ such that $\depth_{\m}(i_1) = \depth_{\m}(j_{c+1})$ and
$e(\Delta_{i_1}) = e(\Delta'_{j_{c+1}})< e(\Delta'_{j_c})\le e(\Delta_{j_c})$.
Thus, $\Delta_{i_1}\subsetneq \Delta_{j_c}$ for otherwise $\Delta_{i_1}\smlr\Delta_{j_c}$
and $\depth_{\m}(j_c)<\depth_{\m}(i_1)=\depth_{\m}(j_{c+1})$ contradicting our assumption.
In particular, by Lemma \ref{lem: ib}\eqref{itm: ib-second}, there exists $i_2\in I$ such that
$\depth_{\m}(i_2) = \depth_{\m}(j_{c+1})$, $\Delta_{j_c}\smlr\Delta_{i_2}$ and $\Delta_{i_2} \subsetneq \Delta_{j_{c+1}}$.
Once again, by the definition of $\m'$, we necessarily have $e(\Delta'_{j_{c+1}})\ge e(\Delta_{i_2})$.
Hence, $e(\Delta_{j_c})\ge e(\Delta'_{j_c})>e(\Delta'_{j_{c+1}})\ge e(\Delta_{i_2})$
in contradiction to the condition $\Delta_{j_c}\smlr\Delta_{i_2}$.

We conclude that $\depth_{\m}(j_0)< \depth_{\m}(j_1)<\ldots < \depth_{\m}(j_k)$, which implies that $k\leq \depth_{\m}(j_k)$
as required.
\end{proof}

We define a map
\[
\Vien: \Mult\setminus\{0\} \to \Lad \times \Mult,\quad \Vien(\m)=(\ldrr{\m},\drv\m)\;.
\]

It is clear that $\Vien(\m)$ is well defined (i.e., does not depend on the choice
of admissible enumerations of the fibers of $\depth$).
We call $\ldrr{\m}$ the \emph{highest ladder} of $\m$ and $\drv\m$ the \emph{derived multisegment} of $\m$.

Define recursively
\[
\RSK: \Mult \to \Lads'
\]
by
\[
\RSK(0)=\emptyset,\ \ \RSK(\m)=(\ldrr{\m},\RSK(\drv\m)),\ \ \m\ne0.
\]
Adapting the discussion of \cite{MR0272654}*{\S4} to our conventions we obtain that
\[
\RSKn(\m) = (P(\RSK(\m)),Q(\RSK(\m))).
\]


\subsubsection{On the image of $\RSK$}
The fact that we consider only multisegments rather than multisets of arbitrary pairs of integers
means that there are restrictions on the image of $\RSK$. This situation is explained in Appendix \S\ref{sec:appendix},
where a full characterization of the image of (a certain variant of) $\RSK$ is given by means of \textit{key tableaux}
in the sense of Lascoux and Sch\"utzenberger, which has links to the theory of Demazure crystals.

A partial characterization from a somewhat different point of view will be given here, and will be relevant for the core of our arguments.
A combined understanding of both these approaches may be beneficial for a better understanding of the entire theme of this paper. We leave it for future endeavors.


Let $\m$ be a multisegment and $\la$ a ladder.
We say that $\la$ is dominant with respect to $\m$, if $\la$ is dominant with respect to any
ladder sub-multisegment of $\m$.
We say that the pair $(\la,\m)$ is permissible
if $(\la,\la')$ is permissible (and in particular, dominant) for any ladder sub-multisegment $\la'$ of $\m$.
Denote by $\pairs\subset \Lad \times \Mult$ the set of permissible pairs.

\begin{proposition}
The map $\Vien$ defines a bijection
\[
\Vien:\Mult\setminus\{0\} \rightarrow\pairs.
\]
Moreover, the image of the map $\RSK$ is contained in $\Lads$.
\end{proposition}

\begin{remark} \label{rem: imrsk}
The map $\RSK$ is not onto $\Lads$.
For instance, $([3,3]+[1,2],[2,3],[1,2])$ is not in the image of $\RSK$
since $\RSK([1,3]+[2,2])=([2,3],[1,2])$ but $([3,3]+[1,2],[1,3]+[2,2])$
is not permissible.

However, if we restrict $\RSK$ to the set of multisegments $\m=\sum_{i\in I}\Delta_i$ such that $b(\Delta_i)\le e(\Delta_j)$ for all $i,j\in I$, then
the image consists of the subset of $\Lads'$ consisting of the tuples $(\la_1,\dots,\la_m)\in \Lads'$
such that $b(\Delta)\le e(\Delta')$ for any $\Delta$ and $\Delta'$ which occur in $\sum\la_i$. We refer to this as the saturated case.
\end{remark}

\begin{proof}
Fix $0\neq \m\in \Mult$ as before.

It follows from Lemma \ref{lem: ib}\eqref{itm: ib-first} and equations \eqref{eq: max} that $\lshft\Delta_j\smlr\Delta'_i$ for any $i\in\ldr{I}$ and $j\in I$ such that
$\depth(j)\ge\depth(i)$. In particular, $\lshft\Delta'_j\smlr\Delta'_i$
if in addition $j\in I'$.

It then follows by Lemma \ref{lem: ic} that $\ldrr{\m}$ is dominant with respect to $\m'$.

Note that $e(\Delta_j)\ge b(\Delta_i)$ for all $i,j\in I$ with $\depth(i) = \depth(j)$. Hence, by \eqref{eq: max} and the fact that $\ldrr{\m}$ is a ladder, we see that
$e(\Delta'_j)\ge b(\Delta'_i)$ for any $i\in\ldr{I}$ and $j\in I'$ such that $\depth(j)\le\depth(i)$.

The fact that $(\ldrr{\m},\m')$ is permissible now follows again from Lemma \ref{lem: ic}. In conclusion, $\Vien(\m)\in\pairs$.

To show that $\Vien$ is a bijection we describe the inverse $\Vein:\pairs\rightarrow\Mult\setminus\{0\}$
following \cite{MR1464693}*{\S4.2}.

Suppose that $\la=\sum_{j\in J}\Delta_i$, $J=\{1,\dots,m\}$ is a ladder, such that
$\Delta_{r+1}\smlr\Delta_r$ for all $r=1,\dots,m-1$.
Let $\m=\sum_{i\in I}\Delta_i$ be a multisegment (taking the index sets $I,J$ as disjoint sets), such that $(\la,\m)$ is permissible.
In particular, $\la$ is dominant with respect to $\m$.

We define $g=g_{\m,\la}:I\rightarrow J$ and $f=f_{\m,\la}:I\rightarrow J$ by
\begin{gather*}
g(i)=\max\{j\in J:\lshft\Delta_i\smlr\Delta_j\},\\
f(i)=\min(g(i),\{f(j)-1:j\in I, \Delta_j\smlr\Delta_i\}).
\end{gather*}
Equivalently,
\[
f(i)=\min\{g(i_k)-k:\exists i=i_0,\dots,i_k\in I\text{ such that }\Delta_{i_{r+1}}\smlr\Delta_{i_r}, r=0,\dots,k-1\}.
\]
By our assumption, $f$ is well defined.
Moreover, for any $j\in J$ we may write the fiber $Y_j=f^{-1}(j)$
as $Y_j=\{i_1,\dots,i_k\}$ (possibly with $k=0$) where $\Delta_{i_{r+1}}\subset\Delta_{i_r}$, $r=1,\dots,k-1$,
$e(\Delta_{i_k})\ge b(\Delta_j)\ge b(\Delta_{i_k})$, $e(\Delta_j)\ge e(\Delta_{i_1})$.
(The condition $e(\Delta_{i_k})\ge b(\Delta_j)$ follows from the permissibility of $(\la,\m)$.)

Let $\sigma$ be the permutation of $I\cup J$ whose cycles are $(i_1,\dots,i_k,j)$ as we vary over $j\in J$.

For any $i\in I\cup J$, we set $\Delta'_i=[b(\Delta_{\sigma(i)}),e(\Delta_i)]$.
Note that we still have $\Delta'_{i_{r+1}}\subset\Delta'_{i_r}$, $r=1,\dots,k-1$
and also $\Delta'_{i_1}\subset\Delta'_j$.

Finally, set $\Vein(\la,\m)=\sum_{i\in I\cup J}\Delta'_i$. It is easy to see that $\Vein$ is the inverse of $\Vien$.

It remains to prove the last statement of the proposition regarding the map $\RSK$.
To that end, we need to show that for any pair $(\la,\m)\in\pairs$ with $\m\ne0$
we have $(\la,\ldrr{\m}), (\la,\drv{\m})\in\pairs$ as well.

Suppose that $(\la,\m)\in\pairs$.
Assume that $\{0,\dots,n\}\cap I=\emptyset$ and write $\la=\sum_{i=0}^n\Delta_i$.
Suppose that $\lshft\Delta'_i\smlr\Delta_j$ for some $i\in\ldr{I}$ and $j=0,\dots,n$
with $j$ maximal.
Let $i_1, i_2\in I$ be such that $\depth(i_1)=\depth(i_2)=\depth(i)$
and $\Delta'_i=[b(\Delta_{i_1}),e(\Delta_{i_2})]$.
Clearly $\lshft\Delta_{i_1}\smlr\Delta_j$ and $\lshft\Delta_{i_2}\smlr\Delta_j$.
Suppose that $j\ne n$. Since $\lshft\Delta'_i\not\smlr\Delta_{j+1}$ we cannot have both
$\lshft\Delta_{i_1}\smlr\Delta_{j+1}$ and $\lshft\Delta_{i_2}\smlr\Delta_{j+1}$.
Let $i'\in\{i_1,i_2\}$ be such that $\lshft\Delta_{i'}\not\smlr\Delta_{j+1}$.
If $j=n$ then take $i'=i_1$ or $i_2$ -- it makes so difference.
In either case, there exist $k_0,\dots,k_s=i'$ with $s=\depth(i)$ such that
$\Delta_{k_r}\smlr\Delta_{k_{r-1}}$, $r=1,\dots,s$ and $\depth(\Delta_{k_r})=
r$ for all $r$.
Since $(\la,\m)$ is permissible, for any $r\le s$
we have $e(\Delta_{k_r})\ge b(\Delta_{j-s+r})$. This implies that
$e(\Delta'_k)\ge b(\Delta_{j-s+r})$ where $k\in\ldr{I}$ is such that
$\depth(k)=r$. It follows that $(\la,\ldrr{\m})$ is permissible.

Next, suppose that $i_0,\dots,i_k\in\drv{I}$ with $\Delta'_{i_r}\smlr\Delta'_{i_{r-1}}$, $r=1,\dots,k$.
We claim that there exist $j_0,\dots,j_k$
such that $\Delta'_{j_r}\smlr\Delta'_{j_{r-1}}$, $r=1,\dots,k$
$\depth(j_r)=\depth(i_r)$ for all $r$ and $e(\Delta_{j_r})\le e(\Delta'_{i_r})$
for all $r$ and $e(\Delta'_{i_k})=e(\Delta_{j_k})$.

We construct $j_r$ by descending induction on $r$. Let $j_k$ be such that
$e(\Delta'_{i_k})=e(\Delta_{j_k})$ and $\depth(i_k)=\depth(j_k)$.
Suppose that $j_r$ was constructed and $r\ge0$. Let $j'$ be such that
$e(\Delta_{j'})=e(\Delta'_{i_{r-1}})$ and $\depth(j')=\depth(i_{r-1})$.
Then $e(\Delta_{j'})>e(\Delta'_{i_r})\ge e(\Delta_{j_r})$.
If $\Delta_{j_r}\smlr\Delta_{j'}$, take $j_{r-1}=j'$. Otherwise,
$\Delta_{j'}\supset\Delta_{j_r}$.
Take $j_{r-1}$ such that $\depth(j_r)=\depth(j')$ and $\Delta_{j_r}\smlr\Delta_{j_{r-1}}$.
Then necessarily $\Delta_{j_{r-1}}\subset\Delta_{j'}$ and hence
$e(\Delta_{j'})\ge e(\Delta_{j_{r-1}})$ as required.

Now let $j\in\{0,\dots,n\}$ be the maximal index such that
$\Delta_{j_k}\smlr\Delta_j$. By the permissibility of $(\la,\m)$
we have $e(\Delta_{j_r})\ge b(\Delta_{j+r-k})$ for all $r$.
On the other hand, $\Delta'_{i_k}\smlr\Delta_j$ and
therefore, if $j'\in\{0,\dots,n\}$ is the maximal index such that
$\Delta'_{i_k}\smlr\Delta_{j'}$, then $j'\ge j$.
Therefore, $e(\Delta'_{i_r})\ge e(\Delta_{j_r})\ge b(\Delta_{j+r-k})\ge b(\Delta_{j'+r-k})$.
It follows that $(\la,\drv{\m})$ is permissible.
\end{proof}

\subsubsection{An inductive description}

We finish our discussion on the RSK algorithm with the following lemma, which allows for inductive arguments in certain cases.

\begin{lemma}\label{lem: sc}
Assume that $\m\ne0$ and $\min\ldrr{\m}=\min\m$.
Let $d=d(\m)$ be the depth of $\m$ and let $i_0\in\depth_{\m}^{-1}(d)$ be such that
$\Delta_{i_0}\supset\Delta_i$ for all $i\in\depth_{\m}^{-1}(d)$.
Then
\begin{enumerate}
\item\label{it: seco} $\depth_{\m}^{-1}(d) = \{i\in I\;:\; b(\Delta_i) = \min \m,\; \Delta_{i_0}\supset\Delta_i\}\;.$
\item $\Delta_{i_0}\le\ldrr{\m}$.
\item $\sum_{i\in\depth_{\m}^{-1}(d)\setminus\{i_0\}}\Delta_i\le\drv\m$.
(It is of course not excluded that $\depth_{\m}^{-1}(d)=\{i_0\}$.)
\item
\[
\Vien\big(\m - \sum_{i\in\depth_{\m}^{-1}(d)}\Delta_i\big)=
\big(\ldrr{\m}-\Delta_{i_0},\drv\m-\sum_{i\in\depth_{\m}^{-1}(d)\setminus\{i_0\}}\Delta_i\big).
\]
\end{enumerate}
\end{lemma}

\begin{proof}
Let $\Delta$ be the segment in $\ldrr{\m}$ with $b(\Delta)=\min\ldrr{\m}$.
By equation \eqref{eq: max}, $b(\Delta_i)=\min\m$ for any $i\in\depth_{\m}^{-1}(d)$.
Similarly, $e(\Delta)=e(\Delta_{i_0})$ and hence, $\Delta=\Delta_{i_0}$.

The equality in \eqref{it: seco} becomes obvious from the fact that $\depth_{\m}(i)\geq \depth_{\m}(i_0)$ for such $i\in I$.

The rest of the lemma follows from the description of $\Vien(\m)$.
\end{proof}

\section{Commutativity of algorithms}\label{sec: comm}
We turn now to study the relations between the M\oe glin--Waldspurger algorithm and RSK.

Let $\m= \sum_{i\in I} \Delta_i \in \Mult$ ba a multisegment, with notation as before. We will consider the more involved case which is not covered by Lemma \ref{lem: sc}, namely, when $\min\m<\min\ldrr{\m}$.

\begin{lemma} \label{lem: pre1}
Suppose that $\m$ is non-degenerate (see \S\ref{sec: ndgnrt}) and let $I^*$ be a set of leading indices for $\m$. (See \S\ref{sec: leading}.) Then
\begin{enumerate}
\item For every $i\in I$ we have $\depth_{\m}(i)\le\depth_{\m^{\MWn}}(i)\le\depth_{\m}(i)+1$.
\item\label{it: change} We write
\[
I_{\spcl}=\{i\in I:\exists i'\in I^*\text{ such that }b(\Delta_i)=b(\Delta_{i'}), \Delta_i\subsetneq\Delta_{i'}
\text{ and }\depth_{\m}(i)=\depth_{\m}(i')\}\;.
\]
Then, $I^*$ and $I_{\spcl}$ are disjoint. The equality $\depth_{\m^{\MWn}}(i)=\depth_{\m}(i)+1$ holds if and only if $i\in I_{\spcl}$.
Moreover, if $i\in I_{\spcl}$ and $i'\in I^*$ are such that $b(\Delta_i)=b(\Delta_{i'})$, $\Delta_i\subsetneq\Delta_{i'}$  and $\depth_{\m}(i)=\depth_{\m}(i')$, then $i'\neq i^*_{\min}$ and
$\depth_{\m^{\MWn}}(i)=\depth_{\m}(\prevind{i'})$.
\item $d(\m)=d(\m^{\MWn})$.
\item $\depth_{\m}$ is injective on $I^*$.
\end{enumerate}
\end{lemma}

\begin{remark}
In general, it is not true that $\depth_{\m}(\prevind{i})=\depth_{\m}(i)+1$
for all $i\in I^*\setminus\{i^*_{\min}\}$.
\end{remark}

\begin{proof}
Let $i\in I$.
To show that $\depth_{\m}(i)\le\depth_{\m^{\MWn}}(i)$ suppose that
$\Delta_{i_1}\smlr\dots\smlr\Delta_{i_k}$ for some $i_1,\dots,i_k\in I$ with $i_1=i$.
We claim that we can choose these indices so that whenever $i_r\in I^*$
for some $r$, either $r=k$ or $b(\Delta_{i_{r+1}})>b(\Delta_{i_r})+1$ or
$b(\Delta_{i_{r+1}})=b(\Delta_{i_r})+1$ and $i_{r+1}\in I^*$.
Indeed, whenever $i_r\in I^*$ with $r<k$ and $b(\Delta_{i_{r+1}})=b(\Delta_{i_r})+1$ with $i_{r+1}\notin I^*$
we may replace $i_{r+1}$ by the leading index $i$ such that $b(\Delta_i)=b(\Delta_{i_{r+1}})$.
Iterating this process we will get the required property. With this extra property we have
$\Delta^*_{i_1}\smlr\dots\smlr\Delta^*_{i_k}$. Thus $\depth_{\m}(i)\le\depth_{\m^{\MWn}}(i)$.

It is clear that $I_{\spcl}\cap I^*=\emptyset$ and that if $i\in I_{\spcl}$ then $\depth_{\m^{\MWn}}(i)>\depth_{\m}(i)$.
Indeed, if $i'\in I^*$ is such that $b(\Delta_i)=b(\Delta_{i'})$, $\Delta_i\subsetneq\Delta_{i'}$
and $\depth_{\m}(i)=\depth_{\m}(i')$ then $\Delta_i^*=\Delta_i\smlr\Delta^*_{i'}$ and therefore
$\depth_{\m^{\MWn}}(i)>\depth_{\m^{\MWn}}(i')\ge \depth_{\m}(i')=\depth_{\m}(i)$.
Moreover, it is clear from the definition of $i^*_{\min}$ that $i'\neq i^*_{\min}$ and since $i'\in I^*$ we have $\Delta_{\prevind{i'}}\supset\Delta_i$. We claim that
$\depth_{\m}(i')+1=\depth_{\m}(\prevind{i'})$. Clearly,
$\depth_{\m}(\prevind{i'})>\depth_{\m}(i')$. If $\depth_{\m}(\prevind{i'})>\depth_{\m}(i')+1$ then there exists $j\in I$ such that
$\Delta_{\prevind{i'}}\smlr\Delta_j$ and $\depth_{\m}(j)>\depth_{\m}(i')$.
If $b(\Delta_j)=b(\Delta_{i'})$ then $\Delta_j\supset \Delta_{i'}$ since
$i'\in I^*$ and we would get a contradiction. Otherwise, $\Delta_i\smlr\Delta_j$
and again we get a contradiction.

It remains to show that $\depth_{\m^{\MWn}}(i)\le\depth_{\m}(i)+1$ for all $i\in I$
with equality only if $i\in I_{\spcl}$.
We prove this by descending induction on $e(\Delta_i)$.

The statement is trivial if $\depth_{\m^{\MWn}}(i)=0$. Suppose that $\depth_{\m^{\MWn}}(i)>0$.
Then there exists $i_1\in I$ such that $\depth_{\m^{\MWn}}(i)=\depth_{\m^{\MWn}}(i_1)+1$ and $\Delta^*_{i}\smlr\Delta^*_{i_1}$.
If $\Delta_i\not\smlr\Delta_{i_1}$ then $i_1\in I^*$, $b(\Delta_i)=b(\Delta_{i_1})$ and $\Delta_i\subsetneq\Delta_{i_1}$.
Our claim follows in this case since $\depth_{\m}(i_1)\le\depth_{\m}(i)$ and
by induction hypothesis $\depth_{\m^{\MWn}}(i_1)=\depth_{\m}(i_1)$.

Assume therefore that $\Delta_i\smlr\Delta_{i_1}$. Then $\depth_{\m}(i)\ge\depth_{\m}(i_1)+1$.
It follows from the induction hypothesis that
\[
\depth_{\m^{\MWn}}(i)=\depth_{\m^{\MWn}}(i_1)+1\le\depth_{\m}(i_1)+2\le\depth_{\m}(i)+1.
\]
Assume that $\depth_{\m^{\MWn}}(i)=\depth_{\m}(i)+1$. Then, again by induction hypothesis $i_1\in I_{\spcl}$, i.e.
there exists $i_2\in I^*$ such that  $b(\Delta_{i_2})=b(\Delta_{i_1})$, $\Delta_{i_1}\subsetneq\Delta_{i_2}$
and $\depth_{\m}(i_1)=\depth_{\m}(i_2)$. In particular $i_2\ne i^*_{\min}$.
Let $i_3=\prevind{i_2}\in I^*$.
Since $i_1\notin I^*$ we must have $e(\Delta_{i_3})\ge e(\Delta_{i_1})$.
Hence $\Delta^*_i\smlr\Delta^*_{i_3}$ and therefore $\depth_{\m}(i_3)\le\depth_{\m^{\MWn}}(i)-1=\depth_{\m}(i)$.
(Note that $\depth_{\m}(i_3)=\depth_{\m^{\MWn}}(i_3)$ by induction hypothesis.)
On the other hand, $\depth_{\m}(i_3)\ge\depth_{\m}(i_2)+1=\depth_{\m}(i_1)+1=\depth_{\m^{\MWn}}(i_1)=\depth_{\m^{\MWn}}(i)-1$.
Thus, $\depth_{\m}(i_3)=\depth_{\m}(i)$. Also, $b(\Delta_{i_3})=b(\Delta_i)$ for otherwise
$\Delta_i\smlr\Delta_{i_3}$ and then $\depth_{\m}(i)\ge\depth_{\m}(i_3)+1$, in contradiction
to what we just proved.

Finally, the last part of the lemma is evident.
\end{proof}

\begin{proposition} \label{lem: main2}
Suppose that $\m\ne0$ and $\min\m<\min\ldrr{\m}$. Then
\begin{enumerate}
\item\label{it: main-one} For any $i\in I^*$, there exists $j \in I$ such that $b(\Delta_j)>b(\Delta_i)$
and $\depth_{\m}(i)=\depth_{\m}(j)$. In particular, $\m$ is non-degenerate.
\item For any $i\in I^*$, let $i^\#\in I$ be the distinguished index (see \S\ref{sec: distin}) such that $b(\Delta_i)=b(\Delta_{i^\#})$,
$\depth_{\m}(i)=\depth_{\m}(i^\#)$ and $e(\Delta_{i^\#})$ is minimal with respect to these properties.
Then,
\[
I_{\sharp}=\{i^\#:i\in I^*\}
\]
is a set of leading indices for $\drv\m$.
In particular, $\Delta^{\nMW}(\m)=\Delta^{\nMW}(\drv\m)$.
\item Assume (as we may) that $I^*$ consists of distinguished indices. In particular, $i=i^\#$
whenever $\Delta_i=\Delta_{i^\#}$, for $i\in I^*$.
Let $\tau$ be the permutation of the index set $I$ defined by $\tau(i)=i$ for all $i\notin I^*\cup I_{\sharp}$,
$\tau(i)=\prevind{i}$ for all $i\in I^*\setminus I_{\sharp}$, and
\[
\tau(i^\#)=\begin{cases}(\nextind{i})^\#&\text{ if }i\ne i^*_{\max}\text{ and }(\nextind{i})^\#\neq \nextind{i} ,\\
i&\text{otherwise}\end{cases}
\]
for all $i\in I^*$.
Then for any $i\in I$,
\[
(\Delta_{\tau(i)}^*)'=\begin{cases}{}^-\Delta_i'&i\in I_{\sharp},\\
\Delta_i'&\text{otherwise.}\end{cases}
\]
\end{enumerate}
\end{proposition}

\begin{proof}
\begin{enumerate}
\item
Assume on the contrary that $i\in I^*$ and $b(\Delta_j)\le b(\Delta_i)$ whenever $\depth_{\m}(j)=\depth_{\m}(i)$.
Assume further that $b(\Delta_i)$ is minimal with respect to this property.

If $i=i^*_{\min}$, then $\depth_{\m}(i)=d(\m)$, because we cannot have $\Delta_j\smlr\Delta_{i'}$,
for any $j\in I$ and $i'\in\depth_{\m}^{-1}(\depth_{\m}(i))$. However, in this case we will get a contradiction to the
assumption that $\min\m<\min\ldrr{\m}$.

Suppose that $i\ne i^*_{\min}$ and let $j\in I$ be any index such that $\depth_{\m}(j)=\depth_{\m}(\prevind{i})$.
Then $\depth_{\m}(j)>\depth_{\m}(i)$ and therefore
there exists $i'\in I$ such that $\Delta_j\smlr\Delta_{i'}$ and $\depth_{\m}(i')=\depth_{\m}(i)$.
By our assumption $b(\Delta_j)<b(\Delta_{i'})\le b(\Delta_i)$
and hence $b(\Delta_j)\le b(\Delta_{\prevind{i}})$.
We get a contradiction to the minimality of $i$.

The non-degeneracy part is clear, since if $b(\Delta_i)=e(\Delta_i)$ had been satisfied for some $i\in I^*$, and
$j\in I$ had been such that $b(\Delta_j)>b(\Delta_i)$, then $\Delta_i\smlr\Delta_j$ would have implied
$\depth_{\m}(i)>\depth_{\m}(j)$.

\item
It follows from part \eqref{it: main-one} that $I_{\sharp}\cap \ldr{I}= \emptyset$.
In particular, $e(\Delta_{(i^\#)^\vee}) \le e(\Delta_{i^\#})$, for all $i\in I^*$.
Also, since $i^\#$ is distinguished, its defining property imposes
\begin{equation}\label{eq: dist}
b(\Delta_{i^\#})< b(\Delta_{(i^\#)^\vee})\;.
\end{equation}
Clearly, $(i^*_{\min})^\# =i^*_{\min}\in I_{\sharp}$.
We first claim that $\Delta'_{i^*_{\min}}$ is the shortest segment of $\drv\m$ which begins at $\min\m$.

Suppose on the contrary that this is not the case.
Then, $\Delta'_i \subsetneq \Delta'_{i^*_{\min}} \subseteq \Delta_{i^*_{\min}}$ for some $i\in I$ with $b(\Delta_i) = \min \m$.
Thus, $e(\Delta_{i^\vee}) = e(\Delta'_i)$ means $\Delta_{i^\vee}\subsetneq \Delta_{i^*_{\min}}$. On the other hand, by the defining property of $i^*_{\min}$, $\Delta_{i^*_{\min}}\subsetneq \Delta_i$ (inequality because of $\Delta'_i \neq \Delta'_{i^*_{\min}}$). Yet, because of $b(\Delta_i) = b(\Delta_{i^*_{\min}})$ we have $\depth_{\m}(i) \le \depth_{\m}(i^*_{\min})$, which now contradicts Lemma \ref{lem: id}.

Now, let $i\in I^*$ with $i\ne i^*_{\max}$ be fixed.
To ease the notation, set $j = i^\#, j' = (\nextind{i})^\#,  k = j^\vee, k' = (j')^\vee\in I$.

We have $\Delta_i\smlr\Delta_{\nextind{i}}$ and therefore
\[
\depth_{\m}(i)=\depth_{\m}(j)=\depth_{\m}(k)> \depth_{\m}(\nextind{i}) = \depth_{\m}(j')= \depth_{\m}(k')\;.
\]

By \eqref{eq: dist}, $b(\Delta_k)\geq b(\Delta_j)+1 = b(\Delta_{j'})$. Thus, $\Delta_k\subsetneq \Delta_{j'}$, since otherwise we would have $\depth_{\m}(j')\geq \depth_{\m}(k)$.
So, by Lemma \ref{lem: id}, we cannot have $\Delta_{k'}\subset \Delta_k$. The depth inequality also forbids the condition $\Delta_{k'} \smlr \Delta_k$.
Hence, we must have $e(\Delta_k)\leq e(\Delta_{k'})$. Now, an equality $e(\Delta_k)= e(\Delta_{k'})$ together with the implied containment
$\Delta_{k}\subset \Delta_{k'}$ would again contradict the depth inequality. Summing up, $e(\Delta_k)< e(\Delta_{k'})$, which means $\Delta'_j \smlr \Delta'_{j'}$.

Next, we prove that with $i,j,k,j',k'\in I$ as before, there does not exist a segment $\Delta$ of $\drv\m$
such that $\Delta'_j\smlr\Delta$, $b(\Delta)=b(\Delta'_j)+1(=b(\Delta_i)+1)$ and
$\Delta\subsetneq\Delta'_{j'}$.

Suppose otherwise. Then such a segment satisfies $\Delta = [b(\Delta_l), e(\Delta_{l^\vee})]$, for an index $l\in I'$.
By the assumptions, $b(\Delta_l)=b(\Delta_i)+1$, $e(\Delta_k)<e(\Delta_{l^\vee})<e(\Delta_{k'})$ and $ e(\Delta_{l^\vee})\leq e(\Delta_{l})$.

In particular, $\Delta_k \subsetneq \Delta_l$. Now, either $\Delta_j \smlr \Delta_l$ or $\Delta_l \subsetneq \Delta_j$.
By applying Lemma \ref{lem: id} in the latter case, we obtain $\depth_{\m}(l) < \depth_{\m}(j)$ in both cases.

If $\Delta_i\smlr \Delta_l$, we set $m=l$. Otherwise, by Lemma \ref{lem: ib}\eqref{itm: ib-first}, there is $m\in I$, such that $\Delta_i\smlr \Delta_m$ and $\depth_{\m}(m) = \depth_{\m}(l)$.
In that case, $e(\Delta_l)\leq e(\Delta_i)<e(\Delta_m)$ forces $b(\Delta_m) = b(\Delta_i)+1$.

By the definition of $\nextind{i}$ we have $\Delta_{\nextind{i}}\subset \Delta_m$, which implies that $\depth_{\m}(\nextind{i}) \geq \depth_{\m}(m)$.

On the other hand, we have either $\Delta_l \smlr \Delta_{k'}$ or $\Delta_{l^\vee} \smlr \Delta_{k'}$ or $\Delta_{l^\vee} \subsetneq \Delta_{k'} \subsetneq \Delta_l$.
In all three case, with Lemma \ref{lem: id} for the latter, we reach a contradiction to the depth inequality.

Finally, set $j_{\max}=(i^*_{\max})^\#$.
We are left to show that there is no segment $\Delta$ of $\drv\m$ such that $\Delta'_{j_{\max}}\smlr\Delta$
and $b(\Delta)=b(\Delta_{i^*_{\max}})+1(=b(\Delta'_{j_{\max}})+1)$.

Assume the contrary. Then, arguing like before, we obtain $m\in I$ with $\Delta_i \smlr \Delta_m$ and $b(\Delta_m) = b(\Delta_{i^*_{\max}})+1$.
This contradicts the defining property of $i^*_{\max}$.
\item
First note that it follows from Lemma \ref{lem: pre1} that
$(\Delta_i^*)'=\Delta_i'$, for all $i\in I$ with $\depth_{\m}(i)\notin\depth_{\m}(I^*)$.

For any $i\in I^*$, let
\[
J_i=\{j\in I:\depth_{\m}(j)=\depth_{\m}(i),b(\Delta_j)=b(\Delta_i)\text{ and }\Delta_j\subsetneq\Delta_i\}.
\]
Thus, $J_i=\emptyset$ if and only if $i=i^\#$ (since both are distinguished).
For convenience we set $J_{\nextind{i}}=\emptyset$ when $i=i^*_{\max}$.
By Lemma \ref{lem: pre1}, we have
\[
\depth_{\m^{\MWn}}^{-1}(t)=\depth_{\m}^{-1}(t)\cup J_{\nextind{i}}\setminus J_i\;,
\]
for $t=\depth_{\m}(i)$.

Let $\{i_1,\dots,i_l\}$ be the admissible enumeration of $\depth_{\m}^{-1}(t)$.
Then the indices of $J_i$ (if non-empty) appear as a contiguous block (in decreasing order of $e(\Delta_j)$, ending with $i^\#$)
right after the occurrence of $i$ (since $i$ is distinguished).
Upon removing the indices of $J_i$ (if any) and inserting instead the indices of $J_{\nextind{i}}$ next to $i$
(again, in decreasing order of $e(\Delta_j)$, ending with $(\nextind{i})^\#$)
we obtain an admissible enumeration $\{i'_1,\dots,i'_{l'}\}$ of $\depth_{\m^{\MWn}}^{-1}(t)$ (with respect to
$\m^{\MWn}=\sum_{i\in I}\Delta_i^*$).

It follows that
\[
(\Delta^*_i)'=\begin{cases}\Delta_{\nextind{i}}'&\text{if }J_{\nextind{i}}\ne\emptyset,\\
{}^-\Delta_{i^\#}'&\text{otherwise,}\end{cases}
\]
while if $J_{\nextind{i}}\ne\emptyset$, then
\[
(\Delta^*_{(\nextind{i})^\#})'={}^-\Delta_{i^\#}'.
\]
For all $i'\in\depth_{\m^{\MWn}}^{-1}(t)\setminus\{i,(\nextind{i})^\#\}$, we have $(\Delta^*_{i'})'=\Delta_{i'}'$.
In particular, $i_l=i'_{l'}$ and $(\Delta^*_{i_l})'=\Delta_{i_l}'$.
\end{enumerate}
The proposition follows.
\end{proof}

\begin{corollary}\label{cor: main}
For any $0\ne\m\in \Mult$ with $\min\m<\min\ldrr{\m}$ we have
\[
(\Vien\times \id)(\MW(\m))=(\id\times \MW)(\Vien(\m))\;.
\]
In other words, $\ldrr{\m} = \ldrr{\m^{\MWn}}$, $\Delta^{\nMW}(\m)=\Delta^{\nMW}(\drv\m)$ and $(\m^{\MWn})' = (\m')^{\MWn}$.
\end{corollary}

\begin{proof}
It follows from Proposition \ref{lem: main2}\eqref{it: main-one} that $I^*\cap \ldr{I}= I_{\sharp}\cap \ldr{I}= \emptyset$.
Hence, $\ldrr{\m} = \ldrr{\m^{\MWn}}$.
The rest of Proposition \ref{lem: main2} shows that $\Delta^{\nMW}(\m)=\Delta^{\nMW}(\drv\m)$ and $(\m^{\MWn})' = (\m')^{\MWn}$.
\end{proof}

\section{Representation theoretic applications}\label{sec: rep}

\subsection{Basics}
For the rest of the paper we fix a non-archimedean local field $F$ with normalized absolute value $\abs{\cdot}$
and consider representations of the general linear groups $\GL_n(F)$, $n\ge0$.
All representations are implicitly assumed to be complex and smooth.

For any segment $\Delta=[a,b]\in \Seg$, we write $Z(\Delta)$ and $L(\Delta)$ for the character $\abs{\det}^{\frac{a+b}2}$ of $\GL_{b-a+1}(F)$
and the Steinberg representation of $\GL_{b-a+1}(F)$ twisted by $\abs{\det}^{\frac{a+b}2}$, respectively.

Normalized parabolic induction will be denoted by $\times$. More precisely, if $\pi_i$ are representations
of $\GL_{n_i}(F)$, $i=1,\dots,k$ and $n=n_1+\dots+n_k$, we write
\[
\pi_1\times\dots\times\pi_k=\operatorname{Ind}_{P_{n_1,\dots,n_k}(F)}^{\GL_n(F)}\pi_1\otimes\dots\otimes\pi_k\;,
\]
where $P_{n_1,\dots,n_k}$ is the parabolic subgroup of $\GL_n$ consisting of upper block triangular matrices
with block sizes $n_1,\dots,n_k$ and $\pi_1\otimes\dots\otimes\pi_k$ is considered as a representation of
$P_{n_1,\dots,n_k}$ via the pull-back from $\GL_{n_1}(F)\times\dots\times\GL_{n_k}(F)$.

Given a multisegment $\m\in\Mult$, we can write it (in possibly several ways) as $\m=\sum_{i=1}^k\Delta_i$,
where for any $i<j$, we have $\Delta_i\not\smlr\Delta_j$.
Then the representations
\[
Z(\m)=\soc(Z(\Delta_1)\times\dots\times Z(\Delta_k))\;,
\]
\[
L(\m)=\soc(L(\Delta_k)\times\dots\times L(\Delta_1))\;,
\]
are both irreducible and, up to equivalence, depend only on $\m$.
(We recall that $\soc$ stands for the socle, i.e., the sum of all irreducible subrepresentations.)
The M\oe glin--Waldspurger involution switches the two, namely
\[
L(\m)\cong Z(\m^\#),\ \ Z(\m)\cong L(\m^\#).
\]
In fact, the map $Z(\m)\mapsto L(\m)$ is a specialization of a functorial duality on the category
of admissible smooth representations of $\GL_n$ (\cites{MR1471867, MR3769724}).

\begin{remark}
More generally, we can fix a (not necessarily unitary) irreducible supercuspidal representation $\rho$ of $\GL_d(F)$ and consider the irreducible representations
\[
Z([a,b]_{\rho})=\soc(\rho\abs{\det}^a\times\dots\times\rho\abs{\det}^b),\ \
L([a,b]_{\rho})=\soc(\rho\abs{\det}^b\times\dots\times\rho\abs{\det}^a)
\]
of $\GL_{(b-a+1)d}(F)$. (When $\rho$ is the trivial character of $\GL_1(F)=F^*$
this coincides with the previous notation.)
We can then define $Z(\m_{\rho})$ and $L(\m_{\rho})$ for any multisegment $\m$ as before. Theorem \ref{thm: main} below and its proof will hold without change.
Given irreducible supercuspidal representations $\rho_1,\dots,\rho_k$
of $\GL_{d_i}(F)$ such that $\rho_i\not\simeq\rho_j\abs{\det}^r$ for all
$i\ne j$ and $r\in\Z$, and any multisegments $\m_1,\dots,\m_k$,
the representation $Z((\m_1)_{\rho_1})\times\dots\times Z((\m_k)_{\rho_k})$
is irreducible. Moreover, by Zelevinsky classification, any irreducible representation of $\GL_n(F)$ can be written uniquely in this form (up to permuting the factors) \cite{MR584084}.
A similar statement holds for $L(\m)$. Therefore, for all practical purposes it is enough to deal with a single $\rho$.
For concreteness we take $\rho$ to be the trivial character of $F^*$, but as was pointed out above this is essentially immaterial.
\end{remark}

\subsection{}
Suppose that $\la$ is a ladder.
Then, for any irreducible representation $\tau$ of $\GL_n(F)$, each of the representations $\soc(Z(\la)\times\tau)$
and $\soc(\tau\times Z(\la))$ is irreducible and occurs
with multiplicity one in the Jordan--H\"older sequence of $Z(\la)\times\tau$ \cite{MR3573961}.
Thus, for any $\m\in \Mult$ and $\la\in \Lad$, there is a multisegment $\socm(\m,\la)\in \Mult$, such that
\[
\soc(Z(\m)\times Z(\la)) \cong Z(\socm(\m,\la)).
\]
A simple recursive algorithm for the computation of $\socm(\m,\la)$,
which relies on the M\oe glin--Waldspurger algorithm, is given in [ibid.].\footnote{Strictly speaking,
we have to pass to the contragredient.}
We recall the result (using the notation of \S\ref{sec: MW}).

\begin{proposition} (\cite{MR3573961}*{Proposition 6.15} and \cite{MR3866895}*{Lemma 3.16}) \label{prop: lm}
Let $0\ne\m\in\Mult$ and $\la\in \Lad$.
\begin{enumerate}
\item Suppose that $\min\la\le\min \m$.
Let $\Delta$ be the unique segment in $\la$ for which $b(\Delta) = \min \la$.
Then,
\[
\socm(\m,\la) = \socm( \m -\n , \la - \Delta ) + \n + \Delta\;
\]
where upon writing $\m = \sum_{i\in I} \Delta_i$,
\[
\n = \sum_{i\in I:b(\Delta_i) = b(\Delta)\text{ and }e(\Delta_i)\le e(\Delta)} \Delta_i\quad \leq \m \;.
\]
\item
Suppose that $\min \m < \min\la$. Then, $\socm(\m,\la)$ is characterized by the condition
\[
\MW(\socm(\m,\la)) = \big(\socm(\m^{\MWn}, \la) , \Delta^{\nMW}(\m)\big)\;.
\]
\end{enumerate}
\end{proposition}

We remark that we also have
\[
\soc(L(\la)\times L(\m)) \cong L(\socm(\m,\la)).
\]
This can be proved by either repeating the argument of \cite{MR3573961} or
using the functorial properties of the Zelevinsky involution.

\subsection{Main result}
We use the notation of \S\ref{sec: Vien}.

\begin{theorem} ~ \label{thm: main}
\begin{enumerate}
\item For any $0\neq \m\in\Mult$ we have
\[
\socm(\drv\m,\ldrr{\m})=\m,\ \ \soc(L(\ldrr{\m})\times L(\drv\m))=L(\m).
\]
\item For any $(\la,\m)\in\pairs$ we have
\[
\socm(\m,\la)=\Vien^{-1}(\la,\m),\ \ \soc(L(\la)\times L(\m))=L(\Vien^{-1}(\la,\m)).
\]
\item Given $0\neq\m\in \Mult$ write $\RSK(\m)=(\la_1,\dots,\la_k)\in\Lads$
and define recursively
\[
\pi_k=Z(\la_k), \quad \pi_i=\soc(\pi_{i+1}\times Z(\la_i)),\quad i=k-1,\dots,1\;.
\]
Then, $\pi_1 \cong Z(\m)$.

Similarly, letting
\[
\pi'_k=L(\la_k),\quad \pi'_i=\soc(L(\la_i)\times\pi'_{i+1}),\quad i=k-1,\dots,1\;,
\]
we have $\pi'_1\cong L(\m)$.

In particular, $Z(\m)$ (resp. $L(\m)$) occurs as a sub-representation of
\begin{equation} \label{def: std}
\std(\m):= Z(\la_k)\times\cdots\times Z(\la_1)\quad \big( \mbox{resp. }\std'(\m):= L(\la_1)\times\cdots\times L(\la_k)\big)\;.
\end{equation}
\end{enumerate}
\end{theorem}

\begin{proof}
First note that the second part is merely a reformulation of the first part and the last part follows
directly from the first two parts.
Thus, it suffices to prove the first part.
Moreover, by the remark following Proposition \ref{prop: lm}, it is enough to prove the assertion
\[
\socm(\drv\m,\ldrr{\m})=\m.
\]
We argue by induction on $|\m|$, using Proposition \ref{prop: lm}.
The base of the induction is trivial.
Suppose that $0\neq \m = \sum_{i\in I} \Delta_i\in\Mult$. We separate into cases.

Suppose first that $\min\ldrr{\m}=\min\m$.
Let $i_0\in I$ be as in Lemma \ref{lem: sc} and set
\[
\n= \sum_{i\in I:\depth_{\m}(i)=d(\m)}\Delta_i \quad\leq \m\;.
\]
Then, by Lemma \ref{lem: sc},
\[
\Vien(\m-\n) = \big(\la(\m)-\Delta_{i_0}, \    \m' - (\n-\Delta_{i_0})\big)\;.
\]
Since $|\m-\n|<|\m|$, the induction hypothesis now implies that
\[
\m-\n = \socm \big( \m' - (\n-\Delta_{i_0}), \la(\m)-\Delta_{i_0}\big) \;.
\]
It follows from the first part of Proposition \ref{prop: lm} and Lemma \ref{lem: sc}\eqref{it: seco} that $\socm(\m', \la(\m)) = \m$.

Suppose now that $\min\m<\min\ldrr{\m}$. By the second part of Proposition \ref{prop: lm} and Corollary \ref{cor: main}, we have
\begin{align*}
\MW(\socm(\m',\la(\m))) & =  \Big(\socm\big((\m')^{\MWn}, \la(\m)\big) , \Delta^{\nMW}(\m')\Big) \\
&= \Big(\socm\big((\m^{\MWn})', \la(\m^{\MWn})\big) , \Delta^{\nMW}(\m)\Big)\;.
\end{align*}
Yet, since $|\m^{\MWn}|<|\m|$, the induction hypothesis implies that the last expression is nothing but $\MW(\m)$.
The result follows from the injectivity of the map $\MW$.
\end{proof}

\begin{remark}
In \cite{10.1093/imrn/rnz006}, the \textit{width} invariant $k=k(\m)$ was defined for every $\m = \sum_{i\in I} \Delta_i\in \Mult$
to be the maximal number of distinct indices $i_1,\dots,i_k\in I$ for which
$\Delta_{i_{r+1}}\subset\Delta_{i_r}$ for all $r=1,\dots,k-1$.
By standard properties of the RSK correspondence, $k(\m)$ is the number of rows in the tableaux of $\RSKn(\m)$.

It was shown in \cite{10.1093/imrn/rnz006} that if there exist $\la_1,\ldots,\la_l\in \Lad$ such that
$Z(\m)$ appears as a subquotient of $Z(\la_1)\times\dots\times Z(\la_l)$, then $l \geq k(\m)$.
This underlines a minimality property of $\std(\m)$.
\end{remark}

\begin{remark} \label{rem: contra}
For $\m=\sum_{i\in I}[a_i,b_i]$ let $\m^\vee=\sum_{i\in I}[-b_i,-a_i]$
so that $Z(\m^\vee)$ is the contragredient of $Z(\m)$.
Then, $\RSK(\m^\vee)$ is related to $\RSK(\m)$ by the Schutzenberger involution
(modified to our conventions).

We also point out that
\[
(\m^\#)^\vee=(\m^\vee)^\#
\]
since $Z(\m)^\vee=Z(\m^\vee)$ and $L(\m^\vee)=L(\m)^\vee$.
\end{remark}

It would be interesting to extend the second part of Theorem \ref{thm: main} to an arbitrary
pair of a ladder $\la$ and a multisegment $\m$.

\section{Enhanced multisegments}
The results of the previous section stay put if we slightly extend the notion of multisegments to include empty segments of the form
$[a+1,a]$, $a\in\Z$.


Since the proofs are essentially the same, we will only state the results in the modified case, indicating the points where the proofs have to be adjusted.

\subsection{Extended definitions}
For an integer $a$, we will write $\delta_a = (a+1,a)\in \Z\times \Z$ and consider it as a \textit{dummy segment}.
As with segments, we write $b(\delta_a) = a+1$ and $e(\delta_a) = a$. We write
\[
\widetilde{\Seg} = \Seg \;\cup\; \{\delta_a\,: a\in \Z\}
\]
for the set of \textit{enhanced segments}.

The relation $\smlr$ is defined on $\widetilde{\Seg}$ in the same fashion as it was defined on $\Seg$.
By our convention, for $\Delta_1,\Delta_2\in \widetilde{\Seg}$, we write $\Delta_1 \subseteq \Delta_2$, if $b(\Delta_2)\leq b(\Delta_1)$ and $e(\Delta_1)\leq e(\Delta_2)$.

Next, we embed the monoid $\N(\Z)$ into $\N(\Z\times\Z)$ by assigning to each $\hat{a} = a_1 + \ldots + a_k \in \N(\Z)$ the \textit{dummy multisegment}
$\partial_{\hat{a}} = \delta_{a_1} + \ldots + \delta_{a_k} \in \N(\Z\times\Z)$. We will often refer to dummy multisegments as elements of $\N(\Z)$ by implicitly referring to this embedding.
We write $\supp\partial_{\hat{a}}$ for the underlying subset $\{a_1,\dots,a_k\}$ of $\Z$ (without multiplicity).

Given a dummy multisegment $\partial$ and an integer $a\in\Z$, we let $n_\partial(a)$ denote the multiplicity of $\delta_a$ in $\dum$.

We will call $\widetilde{\Mult} = \N(\widetilde{\Seg})$ the set of \textit{enhanced multisegments}.
Again taking $\N(\Z)$ as the collection of dummy multisegments, we naturally obtain an identification
\[
\widetilde{\Mult} = \Mult \times \N(\Z)\;.
\]
Using this presentation we will denote an enhanced multisegment as $\m = (\underline{\m}, \partial(\m))\in \widetilde{\Mult}$, where $\underline{\m}$ is the underlying
(genuine) multisegment and $\partial(\m)$ is a dummy multisegment.

We may view $\widetilde{\Mult}$ as a sub-monoid of $\N(\Z\times\Z)$, by taking $\m \mapsto \underline{\m}+ \partial(\m)$.
In particular, $ \RSKn(\m)= (P_{\m},Q_{\m})\in \mathcal{T}$ is well-defined, for all $\m\in \widetilde{\Mult}$.

Since the relation $\smlr$ and the dominance condition remain well-defined for enhanced segments, we may also define $\widetilde{\Lad} \subset \widetilde{\Mult}$ and
$\widetilde{\Lads'} \subset \bigcup_m \widetilde{\Lad}^m$ analogously to $\Lad$ and $\Lads'$.

The results of \S\ref{subsection: implement} all remain valid for enhanced multisegments and, as a result, give an extension of the previously defined map $\Vien$ to
\[
\Vien: \widetilde{\Mult}\setminus\{0\} \to \widetilde{\Lad} \times \widetilde{\Mult},\quad \Vien(\m)=(\ldrr{\m},\drv\m)\;.
\]

Similarly, we define $\RSK: \widetilde{\Mult} \to \widetilde{\Lads'}$ recursively and finally obtain
\[
\RSKn(\m) = (P(\RSK(\m)),Q(\RSK(\m)))\;,
\]
for all $\m\in \widetilde{\Mult}$.

\subsection{Enhanced results}
For any enhanced multisegment $\m=\underline{\m}+ \partial(\m)$ we write
\[
Z(\m)=Z(\underline{\m}),\ \ L(\m)=L(\underline{\m}).
\]
Similarly, if $\la$ is an enhanced ladder and $\m$ is an enhanced multisegment we write
\[
\socm(\la,\m)=\socm(\underline{\la},\underline{\m}).
\]

In the newly defined terms, Theorem \ref{thm: main} can now be extended to the following statement.
\begin{theorem}\label{thm: dummy}
For any enhanced multisegment $\m\in \widetilde{\Mult}$ with non-zero $\underline{\m}$, we have
\[
\socm\left(\drv\m,\ldrr{\m}\right)=\underline{\m},\ \ \soc\left(L\left(\ldrr{\m}\right)\times L\left(\drv\m\right)\right)=L(\m).
\]
\end{theorem}

The point is that in general, $\underline{\drv\m}\ne\drv{\underline\m}$ and $\ldrr{\underline{\m}}\ne\underline{\ldrr{\m}}$.
Hence, Theorem \ref{thm: dummy} is a genuine extension of  Theorem \ref{thm: main}.


For $\m\in \widetilde{\Mult}$ with a non-zero $\underline{\m}$, let us write $\RSK(\m)=(\la_1,\dots,\la_k)\in\widetilde{\Lads'}$.
We then form the representation
\[
\std(\m):= Z(\underline{\la_k}) \times \cdots \times Z(\underline{\la_1})\;.
\]
Viewed differently, we may start from a choice of a multisegment $0\neq \n \in \Mult$ and an auxiliary choice of a dummy multisegment $\partial\in\N(\Z)$.
Taking $(\n,\partial)\in \widetilde{\Mult}$, we define $\Lambda_\partial(\n):=\Lambda(\n,\partial)$.

\begin{corollary}\label{cor: dumm}
For any $0\neq \n\in \Mult$ and $\partial\in \N(\Z)$, the irreducible representation $Z(\n)$ occurs as a sub-representation of $\Lambda_\partial(\n)$.
\end{corollary}
The case $\partial=0$ amounts to Theorem \ref{thm: main}.

\subsubsection{On the proof of Theorem \ref{thm: dummy}}

Essentially the same arguments that appeared in the previous sections for multisegments, continue to hold for the case of enhanced multisegments.
In this manner, the proof of Theorem \ref{thm: dummy} follows the same lines as that of Theorem \ref{thm: main}. Let us give some details on the necessary adjustments.

Lemma \ref{lem: sc} holds for $\m\in \widetilde{\Mult}$, with the tweaked assumption $\min \underline{\ldrr{\m}} = \min \underline{\m}$,
and $d= \underline{d}(\m)$ being the \textit{essential} depth, that is,
\[
\underline{d}(\m) = \max_{i\in I}\{ \depth_{\m}(i)\,:\, \Delta_i\in \Seg\}\;.
\]
The commutativity of the two algorithms which is the culmination of \S\ref{sec: comm} requires us to extend $\MW$ to a map
\begin{equation} \label{eq: newMW}
\widetilde{\MW}:\widetilde{\Mult}\setminus\N(\Z) \to \widetilde{\Mult} \times \Seg\ \ \ \widetilde{\MW}(\m)=(\Delta^{\nMW}(\m),\m^{\MWn})
\end{equation}
as follows. Given $\m=\sum_{i\in I} \Delta_i \in \widetilde{\Mult}$ with $\underline{\m}\neq 0$,
we set $\Delta^{\nMW}(\m) = \Delta^{\nMW}(\underline{\m})$ and define $\m^{\MWn} = \sum_{i\in I} \Delta_i^*$ as in the multisegment case,
but \textit{without} discarding empty segments. In other words, if $\Delta_i^* = {}^-\Delta_i$ according to the original algorithm and
$\Delta_i = [a,a]$, we now set $\Delta_i^* = \delta_a \in \widetilde{\Seg}$ instead of discarding it.
Note that by our convention the leading exponents $I^*$ are taken from $\underline{\m}$ (not from the dummy part).
Iterating this procedure will give rise to $\underline{\m}^\#$ together with a dummy multisegment.

With this algorithm in mind, practically all arguments of \S\ref{sec: comm} remain valid, with the notable exception of Lemma \ref{lem: pre1}\eqref{it: change},
whose statement should be slightly adjusted. We note that our refined definition of $\MW$ allows us to drop the non-degeneracy assumptions in this section.
Finally, Corollary \ref{cor: main} holds for all $\m\in \widetilde{\Mult}$ with $\min \underline{\m} < \min \underline{\ldrr{\m}}$.

The arguments of \S\ref{sec: rep} can now be applied in the enhanced multisegment setting, with $\underline{\m}$ substituting $\m$ in the required places.
In the concluding part of the proof of Theorem \ref{thm: dummy} we use the observations that $(\underline{\drv{\m}})^{\MWn}$
$ = \underline{(\drv{\m})^{\MWn}} = $
$\underline{\drv{(\m^{\MWn})}}$.

\subsection{Relation to standard modules}
While our newly discovered family $\{\Lambda(\m)\}_{\m\in \Mult}$ of RSK-standard modules satisfies some intriguing properties (see next section for a further discussion),
Corollary \ref{cor: dumm} shows that each choice of a dummy multisegment $\partial \in \N(\Z)$, gives rise to yet another family $\{\Lambda_\partial(\m)\}_{\m\in \Mult}$ of representations.
Fixing $\m\in\Mult$ we will show that in a suitable sense, varying $\partial$ allows us to interpolate between the standard modules \`a la Zelevinsky and Langlands.

Recall that for a given a multisegment $\m=\sum_{i=1}^k\Delta_i\in\Mult$, Zelevinsky constructed the standard module
\[
\zeta(\m)= Z(\Delta_1)\times\dots\times Z(\Delta_k)\;,
\]
where $\Delta_i$'s are enumerated so that for any $i<j$, we have $\Delta_i\not\smlr\Delta_j$. Similarly,
\[
\lambda(\m)= L(\Delta_k)\times\dots\times L(\Delta_1)\;,
\]
will stand for the standard module in the sense of Langlands.\footnote{Note that we reversed the order of factors from the usual convention
in order to obtain $L(\m)$ as a subrepresentation, rather than a quotient.}

Recall further the involution $\m\mapsto \m^\#$ on $\Mult$ (see \S\ref{sec: MW}), for which $Z(\m) \cong L(\m^\#)$
and $L(\m)\cong Z(\m^\#)$.
In particular, for each $\m\in \Mult$, $Z(\m)=\soc(\zeta(\m))=\soc(\lambda(\m^\#))$.

Note that viewing a segment $\Delta=[a,b]\in \Seg$ as a singleton multisegment, we have $\Delta^\# = [a,a] + [a+1,a+1] + \ldots+ [b,b]\in \Lad$.

\begin{proposition}\label{prop: std}
Let $\m=\sum_{i\in I}\Delta_i\in \Mult$ be a multisegment consisting of $r$ segments, and let $\dum\in \N(\Z)$ be a dummy multisegment.
\begin{enumerate}
\item\label{it: std1} Suppose that the following two conditions are satisfied.
\begin{enumerate}
\item Whenever $e(\Delta_j)<e(\Delta_i)$ we have
\[
n_\dum(e(\Delta_i))\ge n_\dum(e(\Delta_j))+\#\{j':e(\Delta_j')=e(\Delta_j)\}.
\]
\item For any $i\in I$ and $t$ such that $t<e(\Delta_i)$ we have $n_\dum(t)\le n_\dum(e(\Delta_i))$.
\end{enumerate}
Then, $\Lambda_{\dum}(\m)\cong\zeta(\m)$.

\item\label{it: std2} Suppose that $n_\dum(i)\ge n_\dum(i+1)+r$, for all $\min\m\le i<\max\m$ (when $\m\ne0$).

Then, $\Lambda_{\dum}(\m)\cong\lambda(\m^\#)$.
\end{enumerate}
\end{proposition}

The first part follows from the following elementary lemma by induction.

\begin{lemma}
Let $0\ne\m \in \widetilde{\Mult}$ be an enhanced multisegment, with $\underline{\m}=\sum_{i\in I}\Delta_i$.

\begin{enumerate}
\item Suppose that $n_{\dum(\m)}(e(\Delta_i))>0$, for all $i\in I$. Then,
\begin{enumerate}
\item $\depth_{\m}(i)$ depends only on $e(\Delta_i)$.
\item $\la(\m) = \sum_{a\in\supp\partial(\m)}\delta_a$.
\item $\m'=\m-\la(\m)$.
\end{enumerate}

\item Suppose that $\underline{\m}\ne0$ and let $j$ be the minimum of $e(\Delta_i)$, $i\in I$.
Let $\Delta$ be the segment in $\underline{\m}$ such that $e(\Delta)=j$ and $b(\Delta)$ is maximal with respect to this property.

Suppose that $n_{\dum(\m)}(j')=0$, for all $j'\le j$, but $n_{\dum(\m)}(e(\Delta_i))>0$, for all $i$ such that $e(\Delta_i)>j$. Then,
\begin{enumerate}
\item $\depth_{\m}(i)$ depends only on $e(\Delta_i)$.
\item $\la(\m) = \Delta+\sum_{a\in\supp\partial(\m)}\delta_a$.
\item $\underline{\m'}=\underline{\m}-\la(\m)$.
\end{enumerate}
\end{enumerate}
\qed
\end{lemma}

For the second part of Proposition \ref{prop: std}, we recall the relation $(\m^\#)^\vee=(\m^\vee)^\#$ for multisegments (see Remark \ref{rem: contra}).
This means that instead of \eqref{eq: newMW} we may consider the map
\[
\m\mapsto(\m_{\MWn}, \Delta_{\nMW}(\m))\;,
\]
where $\m_{\MWn}: = ((\m^\vee)^{\MWn})^\vee$ and $\Delta_{\nMW}(\m): =(\Delta^{\nMW}(\m^\vee))^\vee$.
Applying it iteratively will give rise to $\underline{\m}^\#$, together with a dummy multisegment.
(In fact, for genuine multisegments, this was the original description in \cite{MR863522}.)

Hence, the second part of Proposition \ref{prop: std} follows by a simple induction from the following
elementary lemma whose proof is an easy exercise (cf. \cite{MR863522}*{Remarques II.2}).

\begin{lemma}
Let $\m \in \widetilde{\Mult}$ be an enhanced multisegment with $0\neq \underline{\m}=\sum_{i\in I}\Delta_i$.

Suppose that $n_{\dum(\m)}(j)=0$ for all $j\ge\max\underline{\m}$, but $n_{\dum(\m)}(j)>0$ for all $\min \underline{\m} \le j<\max\m$. Then,
\begin{enumerate}
\item For any $i\in I$, $\depth_{\m}(i)=\max\underline{\m}-e(\Delta_i)$, if there exist $i_0,\dots,i_r=i\in I$
with $r=\max\m-e(\Delta_i)$ such that $\Delta_i=\Delta_{i_r}\smlr\dots\smlr\Delta_{i_0}$;
otherwise $\depth_{\m}(i)=\max\m-e(\Delta_i)-1$.
\item $\la(\m) = \Delta_{\nMW}(\m)^\#+\sum_{a:n_{\dum(\m)}(a)>0,a+1<b(\Delta_{\nMW}(\m))}\delta_a$.
\item $\drv{\m}=\m_{\MWn}-\sum_{a:n_{\dum(\m)}(a)>0}\delta_a$.
\end{enumerate}
\qed
\end{lemma}

\begin{remark}
While Proposition \ref{prop: std} (which is a simple combinatorial statement) and Theorem \ref{cor: dumm} give a new perspective
on the M\oe glin--Waldpusrger involution, they do not give a genuinely
new proof of the relation $Z(\m)=L(\m^\#)$, which is one of the main results of \cite{MR863522}.
The reason is that the techniques and ideas of [ibid.] are used in the results of \cite{MR3573961} and hence indirectly
in the proof of Theorem \ref{cor: dumm}.
\end{remark}

\section{Odds and ends} \label{sec: last}

\subsection{Socle irreducibility}
It is natural to ask whether as in the case of the traditional standard modules, the
newly constructed class of RSK-standard modules possesses the property of having a unique irreducible sub-representation.
Although we are currently unable to prove this in general, we expect that this is indeed the case, namely

\begin{conjecture} \label{conj: main}
Let $\m$ be a multisegment and let $\std(\m)$ be as in \eqref{def: std}. Then $\soc(\std(\m))$ is irreducible, hence (by Theorem \ref{thm: main}), $Z(\m)\cong \soc(\std(\m))$.
\end{conjecture}

A weaker form of this conjecture would be the following

\begin{conjecture} \label{conj: od}
Suppose that $\RSK(\m)=(\la_1,\dots,\la_k)$.
Define recursively $\pi'_1=Z(\la_1)$, $\pi'_i=\soc(Z(\la_i)\times \pi'_{i-1})$, $i=2,\dots,k$.
Then $Z(\m)\cong \pi'_k$.
\end{conjecture}

Note that Conjecture \ref{conj: od} is not a formal consequence of Theorem \ref{thm: main}
since in general it is not true that $\soc(\soc(\pi_1\times\pi_2)\times\pi_3)\simeq
\soc(\pi_1\times\soc(\pi_2\times\pi_3))$, even if $\pi_i$ are supercuspidal.
For instance, we can take $\pi_1=\pi_3$ to be the trivial character of $F^*$ and $\pi_2$ to be
the absolute value on $F^*$.

It is tempting to attempt to prove Conjecture \ref{conj: od} by the same method as Theorem \ref{thm: main}.
Suppose that $\RSK(\m)=(\la_1,\dots,\la_k)$.
Call $\la_k$ the lowest ladder of $\m$ and write $V(\m)=(\la_k,{}'\m)$
where ${}'\m=\RSK^{-1}(\la_1,\dots,\la_{k-1})$ (the fact that it is well defined may follow from arguments in the spirit of \S\ref{sec:appendix}).
We need to show that
\[
Z(\m)=\soc(Z(\la_k)\times Z({}'\m))\;.
\]
As before, it is natural to use the recipe of \cite{MR3573961}*{\S6.3}.
The simple case is when $\max\la_k=\max\m$.
Assume that $\max\la_k<\max\m$. In this case we have to show that
\begin{equation} \label{eq: vmk}
V(\m^{\MWn'})=(\la_k,({}'\m)^{\MWn'})
\end{equation}
where $\m^{\MWn'}$ is the analogue of $\m^{\MWn}$ for the end points.

\subsection{Triangular structure}\label{sect: tri}

For any integer $r$ and an inverted Young tableau $Y$, let $Y_{\ge r}$
be the part of $Y$ consisting of the entries that are bigger than or
equal to $r$. Clearly, $Y_{\ge r}$ is also an inverted Young tableau (of possibly smaller size). Recall the dominance order on the set of Young diagrams defined by
\[
(\lambda_1,\dots,\lambda_k)\dom (\lambda_1',\dots,\lambda'_{k'})\text{ if }
k\le k'\text{ and }\sum_{i=1}^j\lambda_i\ge\sum_{i=1}^j\lambda'_j\text{ for all }
j=1,\dots,k.
\]
(We will only compare Young diagrams of the same size. In this case $\dom$ encodes the
closure relation of unipotent orbits in $\GL_n(\C)$, parameterized by partitions via the Jordan normal form.)
Define a partial order on the set of inverted Young tableaux by
\[
Y\le Y'\quad \text{ if }\sh(Y_{\ge r})\dom\sh(Y'_{\ge r})\;\text{ for all }r\in\Z\;,
\]
where $\sh(X)$ is the shape of $X$, i.e., its underlying Young diagram.
(We will only compare inverted Young tableaux whose entries coincide
as multisets.)
The product partial order on $\Tabs$ induces a partial order
on $\Lads$ (which will be denoted by $\le$), according to the identifications of \S\ref{sect: lad}.

\begin{conjecture}\label{conj: triang}
Let $\m$ be a multisegment and let $\std(\m)$ be as in \eqref{def: std}. In the Grothendieck group we have
\[
\left[\std(\m)\right]=\left[Z(\m)\right]+\sum_{i=1}^l\left[ Z(\n_i)\right]\;,
\]
where $\RSK(\n_i)<\RSK(\m)$ for all $i=1,\dots,l$.\footnote{The $\n_i$'s are not necessarily distinct.
Also, not all $\n$'s with $\RSK(\n)<\RSK(\m)$ necessarily occur.}
\end{conjecture}

We verified this conjecture by computer calculation for all multisegments consisting of $n$ segments with $n\le 8$.
The computation involves writing $Z(\la_k)\times\dots\times Z(\la_1)$ in terms of standard modules (\cite{MR3163355}) and decomposing standard modules into
irreducible representations -- the multiplicities are given by the value at $1$ of Kazhdan--Lusztig polynomials with respect to the symmetric group $S_n$.

We do not know whether the partial order on $\Mult$ (or on the symmetric group for that matter)
given by $\m_1\le\m_2\iff\RSK(\m_1)\le\RSK(\m_2)$ has already been considered in the literature.
Likewise, we are so far unaware of a befitting geometric interpretation or a simpler combinatorial description of this partial order.

Let $\mathcal{R}_n$, $n\ge0$ be the Grothendieck groups of $\GL_n(F)$, $n\ge0$
and let $\mathcal{R}'$ be the subgroup of $\oplus_{n\ge0}\mathcal{R}_n$ generated by $[Z(\m)]$, $\m\in\Mult$.

Conjecture \ref{conj: triang} would imply the following weaker conjecture.

\begin{conjecture}\label{conj: basis}
The classes of RSK-standard representations
\[
\std(\m),\ \m\in\Mult,
\]
form a $\Z$-basis for $\mathcal{R}'$.
\end{conjecture}

\subsection{Relation to the DRS basis}\label{sec: rota}

Recall that $\mathcal{R}'$, equipped with the parabolic induction product, becomes a commutative ring, which is freely generated
by the variables $\{[Z(\Delta)]\}_{\Delta\in \Seg}$.
With these generators, the classes of standard representations become the monomial basis for $\mathcal{R}'$ (as a $\Z$-module).

Let $\overline{a}, \overline{b}\subset \Z$ be two given subsets of integers.
Let $\Mult_{\overline{a},\overline{b}}$ be the collection of multisegments $\m$,
for which the begin points (resp. end points) of all segments comprising $\m$ belong to $\overline{a}$ (resp. $\overline{b}$).
Let $\mathcal{R}_{\overline{a},\overline{b}}$ be the subgroup of $\mathcal{R}'$ generated by $[Z(\m)]$,
$\m\in\Mult_{\overline{a},\overline{b}}$.
In other words, $\mathcal{R}_{\overline{a},\overline{b}}$ is the subring of $\mathcal{R}'$ generated by
the classes of $Z([a,b])$ where $a\in\overline{a}$ and $b\in\overline{b}$ and $a\le b$.
By the properties of $\RSKn$, we have $[\std(\m)]\in \mathcal{R}_{\overline{a},\overline{b}}$, for all
$\m\in \Mult_{\overline{a},\overline{b}}$.
Clearly, Conjecture \ref{conj: basis} is equivalent to the statement that $\std(\m)$, $\m\in\Mult_{\overline{a},\overline{b}}$,
is a $\Z$-basis for $\mathcal{R}_{\overline{a},\overline{b}}$, for all $\overline{a},\overline{b}$.

Now, assume that $\overline{a},\overline{b}$ are such that, $a\leq b$, for every $a\in \overline{a}$ and $b\in \overline{b}$
(the so-called saturated case -- cf. Remark \ref{rem: imrsk}).
Then, $\mathcal{R}_{\overline{a},\overline{b}}$ can be identified with the coordinate ring of the space of matrices of size $|\overline{a}|\times |\overline{b}|$.
Bases for such polynomial rings were considered in \cites{MR0498650, MR0485944} in the context of combinatorial invariant theory.
Our class of RSK-standard modules inside the rings $\mathcal{R}_{\overline{a},\overline{b}}$ becomes precisely the bitableaux basis considered in [ibid.].
Thus, with the above identifications in place, we obtain a weak form of Conjecture \ref{conj: basis}.
\begin{proposition} (cf. \cites{MR0498650, MR0485944, 1605.06696}) \label{prop: basis-rota}
For $\overline{a}, \overline{b}\subset \Z$ in the saturated case, the classes
\[
[\std(\m)],\ \m\in\Mult_{\overline{a},\overline{b}},
\]
form a $\Z$-basis for $\mathcal{R}_{\overline{a},\overline{b}}$.
\end{proposition}

\begin{proof}
Under our assumptions, $\RSK$ gives a bijection between $\Mult_{\overline{a},\overline{b}}$ and the set of pairs of tableaux $(P,Q)$ in $\mathcal{T}$,
for which $P$ (resp. $Q$) is filled by numbers from $\overline{a}$ (resp. $\overline{b}$) (see Remark \ref{rem: imrsk}). Thus, it is enough to prove that the set of elements
\[
\mathcal{B} = \left\{ [Z(\la_1)]\cdot\ldots\cdot[Z(\la_m)]\;:\; (\la_1,\ldots,\la_m)\in \Lads',\; \la_i\in \Mult_{\overline{a},\overline{b}} \right\}
\]
gives a basis for $\mathcal{R}_{\overline{a},\overline{b}}$.

For a ladder $\la = \sum_{i=1}^k [a_i,b_i]$, we have the following determinantal identity \cite{MR3163355}
\[
[Z(\mathfrak{l})] = \sum_{\sigma\in S_k} \epsilon(\sigma) \prod_{i=1}^k [Z([a_i, b_{\sigma(i)}])]\;,
\]
as elements of $\mathcal{R}'$. Here, $\epsilon(\sigma)$ is the sign of the permutation $\sigma$.

Identifying $[Z([a_i, b_j])]$ with variables in the polynomial ring $\mathcal{R}_{\overline{a},\overline{b}}$, we see that our construction of $\mathcal{B}$ coincides with the
construction of the basis in \cite{MR0485944}*{\S4}.
\end{proof}

\appendix

\section{Multisegments, MW, and \texorpdfstring{RSK\\by Mark Shimozono\footnote{Department of Mathematics,
Virginia Polytechnic Institute and State University; \email{mshimo@math.vt.edu}}}{}}
\label{sec:appendix}

\newcommand{\fixit}[1]{{\texttt{\color{red} {*** #1 ***}}}}
	
This appendix contains two main combinatorial results.
The first result, Proposition \ref{P:GL}, describes the effect of the map $\MW$ on the cRSK tableau pair.
The second result, Theorem \ref{T:triangular key}, is a characterization of the image of the set of multisegments under RSK expressed in terms of the right and left key tableau construction of Lascoux and Sch\"utzenberger, applied to the RSK tableau pair.
The complete proofs will appear in a separate publication.

In this appendix, to conform with the bulk of the combinatorics literature, semistandard Young tableaux are used; in contrast, in the main part of this article, inverted Young tableaux
are used; see Remark \ref{R:inverted Young tableau}.

\subsection{Partitions and skew shapes}\label{sec: shapes}
Let $\bY$ be Young's lattice of partitions.
The diagram $D(\lambda)$ of $\lambda\in\bY$ is the set $D(\lambda)=\{(i,j)\in\Z_{>0}^2\mid 1\le j\le\lambda_i\}$. The $(i,j)$-th box is in the $i$-th row and $j$-th column with matrix-style indexing. The transpose or conjugate partition $\lambda^t\in\bY$ is defined by $(i,j)\in D(\lambda)$ if and only if $(j,i)\in D(\lambda^t)$. The notation $\mu\subset\lambda$ means $D(\mu)\subset D(\lambda)$. If $\mu\subset\lambda$ let
$\lambda/\mu:=D(\lambda)\setminus D(\mu)$; this is called a skew shape. The skew shape refers to this set difference of boxes, and can be achieved by various pairs $(\lambda,\mu)$.
A skew shape is \textit{normal} (resp. antinormal) if it has a unique northwestmost (resp. southeastmost) corner.
A skew shape is a horizontal (resp. vertical) strip if it has at most one box in each column (resp. row). Let $\cH$ (resp. $\cV$) denote the set of horizontal (resp. vertical) strips and let $\cH_r$ (resp. $\cV_r$) be the set of horizontal (resp. vertical) strips of size $r$.

\subsection{Skew tableaux}
In this section a semistandard tableau of shape $\lambda/\mu$ is a function $T:\lambda/\mu\to\Z_{>0}$ which is weakly increasing within rows going from left to right and strictly decreasing from bottom to top. Let $\Tab_{\lambda/\mu}$ denote the set of semistandard tableaux of shape $\lambda/\mu$.
A \textit{tableau} (resp. antitableau) is a semistandard skew tableau of normal (resp. antinormal) shape. More precisely, by a skew tableau we mean a translation equivalence class of such maps.
The (row-reading) word of a skew tableau $T$ is defined by
$\word(T) = \dotsm u^{(2)} u^{(1)}$ where $u^{(i)}$ is the word obtained by reading the $i$-th row of $T$ from left to right. Using the row reading word a skew tableau may be regarded as a word. A \textit{tableau} (resp. antitableau) \textit{word} is a word of the form $\word(T)$ where $T$ is a tableau (resp. antitableau).

\subsection{Knuth equivalence}
Let $\equiv$ be Knuth's equivalence relation on the set of words $\bA^*$ with symbols in the totally ordered set $\bA = \{1<2<\dotsm<n\}$
\cite{MR0272654}.
It is the transitive closure of relations $uxzyv\equiv uzxyv$ for $x\le y<z$ and $uyxzv\equiv uyzxv$ for $x<y\le z$ where $x,y,z\in\bA$ and $u,v\in\bA^*$.

\begin{proposition} \cite{MR0272654}
Every Knuth class in $\bA^*$ contains a unique tableau word and a unique antitableau word.
\end{proposition}

Denote by $\bP(u)$ the unique tableau whose word is Knuth equivalent to $u$. This is the Schensted $P$-tableau. For a tableau $T$ denote by $T^\anti$ the unique antitableau whose word is Knuth equivalent to the word of $T$. Then $\bP(u)^\anti$ is the unique antitableau word that is Knuth equivalent to $u$.

\begin{remark} \label{R:inverted Young tableau}
\begin{enumerate}
\item
An inverted Young tableau of \S \ref{sec: Vien} is the same thing as
an antitableau except written a different way: they are mapped to each other by
reflection across a northeast to southwest antidiagonal. See further \S\ref{sec: the RSKs}.
\item Passing from a tableau $T$ to its antitableau $T^\anti$ is a nontrivial computation.
Given the word of $T^\anti$, taking its reverse complement with respect to some ambient alphabet containing all appearing entries, directly gives the word of a tableau known as the Sch\"utzenberger involution of $T$.
\end{enumerate}
\end{remark}

\subsection{Pieri bijections}
A \textit{row word} (resp. \textit{column word}) is a weakly increasing (resp. strictly decreasing) word. Let $\Row$ (resp. $\Col$) denote the set of row (resp. column) words and $\Row_r$ (resp. $\Col_r)$ those of length $r$.

\begin{proposition} \label{P:Pieri} \cite{MR646486}*{Th\'eor\`eme 2.10}
Let $\mu\in\bY$ and $r\in\Z_{\ge0}$.
There are bijections
\begin{align}
\label{E:left row Pieri}
  \Row_r \times \Tab_\mu &\cong \bigsqcup_{\substack{\lambda\in\bY \\ \lambda/\mu\in\cH_r}} \Tab_\lambda &\qquad (u, T) &\mapsto \bP(uT) \\
\label{E:left column Pieri}
  \Col_r \times \Tab_\mu &\cong \bigsqcup_{\substack{\lambda\in\bY \\ \lambda/\mu\in\cV_r}} \Tab_\lambda &  (u, T) &\mapsto \bP(uT) \\
  \label{E:right row Pieri}
 \Tab_\mu\times \Row_r &\cong \bigsqcup_{\substack{\lambda\in\bY \\ \lambda/\mu\in\cH_r}} \Tab_\lambda &  (T,u) &\mapsto \bP(Tu) \\
 \label{E:right column Pieri}
   \Tab_\mu\times \Col_r &\cong \bigsqcup_{\substack{\lambda\in\bY \\ \lambda/\mu\in\cV_r}} \Tab_\lambda &  (T,u) &\mapsto \bP(Tu).
\end{align}
\end{proposition}

\begin{example} \label{X:left column factor}
	We illustrate how to compute the inverse of the bijection \eqref{E:left column Pieri}. Let $\lambda=(4,3,3,2,1)$ and $\mu=(3,2,2,1,1)$. Let $U$ be the unique skew standard tableau of shape $\lambda/\mu\in \cV$, in which the boxes are labeled in increasing value from top to bottom. Pictured below is a tableau $X$ of shape $\lambda$. We perform successive Schensted reverse column insertions starting at the boxes of $U$ starting with the bottommost.
Each reverse insertion produces an additional single entry factored to the left: the bumping path along which entries shift to the left, is depicted in green.
\begin{align*}\ytableausetup{smalltableaux}
	U = \begin{ytableau}
		{}&{}&{}&1 \\ {}&{}&2\\  {}&{}& 3 \\ &4 \\ {}
	\end{ytableau} \quad
X=\begin{ytableau} 1&2&3&3\\2&3&4\\4&4&5\\5&*(green) 6\\ *(green)6 \end{ytableau} \equiv \begin{ytableau} \none&1&2&3&3\\ \none&2&3&4\\ \none&*(green)4&*(green)4&*(green)5\\ \none&5 \\ \none&6 \\ 6\end{ytableau}
	\equiv \begin{ytableau} \none&1&2&3&3\\ \none&*(green)2&*(green)3&*(green)4\\
		\none&4&5 \\ \none&5 \\ \none&6 \\ 4 \\ 6\end{ytableau}
	\equiv \begin{ytableau} \none&*(green)1&*(green)2&*(green)3&*(green)3\\ \none&3&4\\
		\none&4&5 \\ \none&5 \\ \none&6 \\ 2  \\ 4 \\ 6\end{ytableau}
	\equiv \begin{ytableau} \none&2&3&3 \\ \none&3&4 \\ \none&4&5\\ \none &5 \\ \none&6 \\ 	1\\2\\4\\6\end{ytableau} = u T.
\end{align*}
\end{example}

\begin{example} \label{X:right column factor} We illustrate how to compute the inverse of the bijection \eqref{E:right column Pieri}. We use the same $\lambda$, $\mu$, and $U$ as before but a different tableau $Y$. We perform successive Schensted reverse row insertions starting at the boxes of $U$ starting with the bottommost. Each reverse insertion produces an additional single entry factored to the right: the bumping path along which entries shift upwards, is depicted in green.
	\begin{align*}\ytableausetup{aligntableaux=center}
		Y = \begin{ytableau} \none \\ \none \\ \none \\ \none \\
			1&1&*(green)1&2 \\
			2&2&*(green)2 \\
			3&3&*(green)3 \\
			4&*(green)6 \\
			5
		\end{ytableau}
		\equiv \begin{ytableau} \none \\ \none \\ \none \\
			\none&\none&\none&\none&1 \\
			1&1&2&*(green)2 \\
			2&2&*(green)3\\
			3&3&*(green)6 \\
			4\\5\end{ytableau}\equiv
		\begin{ytableau} \none \\ \none \\
			\none&\none&\none&\none&1 \\
			\none&\none&\none&\none&2 \\
			1&1&2&*(green)3 \\
			2&2&*(green)6\\
			3&3 \\
			4\\5\end{ytableau}\equiv
		\begin{ytableau} \none \\
			\none&\none&\none&\none&1 \\
			\none&\none&\none&\none&2 \\
			\none&\none&\none&\none&3 \\
			1&1&2&*(green)6 \\
			2&2\\
			3&3 \\
			4\\5\end{ytableau}\equiv
		\begin{ytableau}
			\none&\none&\none&1 \\
			\none&\none&\none&2 \\
			\none&\none&\none&3 \\
			\none&\none&\none&6 \\
			1&1&2 \\
			2&2\\
			3&3 \\
			4\\5\end{ytableau} = Tu.
		\end{align*}
\end{example}

\subsection{Column RSK correspondence}
Let $\Row^*$ be the set of infinite sequences of row words
$\ub = (\dotsm, u^{(2)},u^{(1)})$ with $u^{(i)}\in \Row$, only finitely many of which are nonempty.

Given $\ub\in\Row^*$ let $P_0 = \emptyset$ be the empty tableau and for $j\ge1$ let $P_j = \bP(u^{(j)} P_{j-1})$. Let $\bP(\ub)=\bP(\dotsm u^{(2)} u^{(1)})$, the
Schensted $P$-tableau of the juxtaposition of the row words $u^{(i)}$.
Note that $\shape(P_j)/\shape(P_{j-1}) \in \cH$ by Proposition \ref{P:Pieri} \eqref{E:left row Pieri}. This sequence of horizontal strips defines a semistandard tableau $\bQ(\ub)$
of the same shape as $\bP(\ub)$ in which the boxes of the horizontal strip $\shape(P_j)/\shape(P_{j-1})$ are filled with $j$'s for all $j$. This yields a bijection we denote by cRSK, the column insertion Robinson-Schensted-Knuth correspondence:
\begin{align}
\begin{diagram}
	\node{\Row^*} \arrow{e,t}{\cRSK} \node{\bigsqcup_\lambda\, \Tab_\lambda \times \Tab_\lambda}  \\
	\node{\ub} \arrow{e,T}{} \node{(\bP(\ub),\bQ(\ub))}
\end{diagram}
\end{align}

\begin{example} \label{X:PQ}
	Let $\ub = (\dotsm,56,34,233,122,12,1)$. We have $P_0$ empty,
	\begin{align*}
	\ytableausetup{smalltableaux}
	P_1&=\begin{ytableau}1\end{ytableau}&\quad P_2&=\begin{ytableau} 1&1\\2\end{ytableau} &\quad P_3&=\begin{ytableau}1&1&1&2\\2&2\end{ytableau} \\
	P_4&=\begin{ytableau}1&1&1&2\\2&2&2\\3&3\end{ytableau}&
	P_5&=\begin{ytableau}1&1&1&2\\2&2&2\\3&3&3\\4\end{ytableau} \\
	\bP(\ub)=P_6&=\begin{ytableau} 1&1&1&2\\2&2&2\\3&3&3\\4&6\\5\end{ytableau}&
	\bQ(\ub)&=\begin{ytableau} 1&2&3&3\\2&3&4\\4&4&5\\5&6\\6\end{ytableau}
	\end{align*}
\end{example}


%

\subsection{Comparison with multisegment notation}\label{sec: the RSKs}
Without loss of generality in this discussion we only consider segments which are intervals of positive integers. Define the injective map $\iota: \mathfrak{M} \to \Row^*$ such that $\iota(\m)=\ub=(\dotsm u^{(2)},u^{(1)})$ where the word $u^{(j)}$ contains the value $i$ the same number of times the multisegment $[i,j]$ appears in $\m$.
The image of $\iota$ is the set of \textit{flagged} tuples of row words, those such that all values appearing in $u^{(i)}$ are less than or equal to $i$.

Let $\flip$ be the reflection across an antidiagonal. There is an involutive bijection
$\invert: \Tab_\lambda \to \ITab_{\lambda'}$ from the set of semistandard tableaux of shape $\lambda$, to the set $\ITab_{\lambda'}$ of inverted tableaux of the conjugate or transposed shape $\lambda'$, defined by $\invert(T) = \flip(T^\anti)$. That is, first take the antitableau and then flip.
\begin{align*}
\begin{ytableau}
	1&1&4\\
	2&3
\end{ytableau} \to
\begin{ytableau}
	\none&1&1\\2&3&4
\end{ytableau}\to
\begin{ytableau}
	4&1 \\
	3&1 \\
	2
\end{ytableau}
\end{align*}
The following diagram commutes:
\begin{align*}
\begin{diagram}
\node{\mathcal{M}} \arrow{e,t}{\RSKn} \arrow{s,t}{\iota} \node{\bigsqcup_\lambda\, \ITab_\lambda \times \ITab_\lambda} \arrow{s,b}{\invert\times\invert} \\
\node{\Row^*} \arrow{e,b}{\mathrm{cRSK}} \node{\bigsqcup_{\lambda} \Tab_\lambda\times\Tab_\lambda}
\end{diagram}
\end{align*}

Let $\RSKn(\m)=(P_{\m},Q_{\m})$ and $\Vien(\m)=(\la(\m),\drv{\m})$. Then
\begin{itemize}
\item $\la(\m)$ is the ladder $[a_1,b_1]+[a_2,b_2]+\dotsm+[a_r,b_r]$ where
$a_1a_2\dotsm a_r$ and $b_1b_2\dotsm b_r$ are the first rows of
$P_{\m}$ and $Q_{\m}$ respectively.
\item $\drv\m$ is uniquely specified by the condition that
$P_{\drv\m}$ and $Q_{\drv\m}$ are obtained from $P_{\m}$ and $Q_{\m}$ by removing their respective first rows.
\end{itemize}

Let $\iota(\m)=\ub$ so that by the commutativity of the above diagram, $\bP(\ub)=\invert(P_{\m})$ and $\bQ(\ub)=\invert(Q_{\m})$.
Consider the map $\Row^* \to \Lad \times \Row^*$ mapping $\ub$ to $(\la(\ub), \drv{(\ub)})$
commuting the diagram
\begin{align*}
\begin{diagram}
	\node{\mathfrak{M}} \arrow{e,t}{\Vien} \arrow{s,t}{\iota} \node{\Lad\times\mathfrak{M}}\arrow{s,b}{\id\times\iota}  \\
\node{\Row^*} \arrow{e,b}{} \node{\Lad\times \Row^*}	
\end{diagram}
\end{align*}
By definition, the ladder $\la(\ub)$ is obtained by the pair of column words
coming from the rightmost rows of the antitableaux
$\bP(\ub)^\anti$ and $\bQ(\ub)^\anti$. Rather than computing the full antitableaux
there is a more efficient way to compute $\la(\ub)$ and $\drv{(\ub)}$.
Let $\bP(\ub) \mapsto (T, u)$ be the application of the inverse of the bijection \ref{E:right column Pieri} for $\lambda=\shape(\bP(\ub))$ and $\mu$ obtained by removing the first column of $\lambda$. Similarly $\bQ(\ub)\mapsto (U, v)$. Then $(u,v)$ is the pair of column words defining $\la(\ub)$. Moreover $\drv{(\ub)}$ is uniquely defined by the property that $\bP(\drv{(\ub)}) = T$ and $\bQ(\drv{(\ub)})=U$.

\begin{remark} There is a version of $\Vien$ which, instead of producing the last
	columns of the antitableaux of $\bP(\ub)$ and $\bQ(\ub)$, produces the first columns of $\bP(\ub)$ and $\bQ(\ub)$.
\end{remark}

\subsection{The map $\MW$ and cRSK}
Let $\ub\in \Row^*$ be upper triangular and nonempty with corresponding multisegment $\fm$. Let $\ubd\in\Row^*$ correspond to $\fm^\dagger$ and let
$\Delta^{\nMW}(\ub)=\Delta^{\nMW}(\mathfrak{m})$.

Given a nonempty word $u$ and $m\in \Z_{>0}$, let $U_m(u)$ be the word obtained by scanning $u$ from right to left, selecting the rightmost copy of $m$ in $u$,
then the next $m+1$ to its left, the next $m+2$ to its left, and so on, until the left end of $u$ is encountered, and then replacing each selected letter $i$ by $i+1$.
It can be shown that $U_m$ can be defined on $\Tab_{\lambda/\mu}$
via the row reading word and on  and on $\Row^*$ via the map $\ub\mapsto \dotsm u^{(2)} u^{(1)}$. Given $\ub\in\Row^*$ nonempty, let $m=\min(\ub)$ be the minimum value among the words in $\ub$. Denote by $\kh=\kh(\ub)\in\Z_{\ge0}$ the maximum index such that $m+j\in u^{(m+j)}$ for all $0\le j<\kh$.
Define $\Uh(\ub)\in\Row^*$ by removing a copy of $m+j$ from $u^{(m+j)}$ for all $0\le j<\kh$.

We briefly define the algorithm $\MW$ on $\ub$ and describe its affect on the cRSK tableau pair. Let $\min(\ub)$ be the minimum value among the words in $\ub$.
We first consider $\Uh(\ub)$ with $m=\min(\ub)$ and $\kh=\kh(\ub)$. Then
$\ubd = U_{m+\kh}(\ub)$. Let $k=k_m(\ub)$ be the sum of $\kh$ and the number of values incremented by $U_{m+\kh}$. By definition $\Delta^{\nMW}(\ub)=[m,m+k-1]$.

\begin{example}\label{X:MW} With $\ub$ as in Example \ref{X:PQ} we have $\Uh(\ub) = (\dotsm, 56,34, 233, 122,1,\emptyset)$, $m=1$,
	$\kh=2$, $\ubd = U_3(\Uh(\ub)) = (\dotsm, 66,35, 234, 122,1,\emptyset)$ and $k=5$.
\end{example}

Given a tableau $T$ and a subtableau $S$, both of partition shape, let $T-S$ be the skew tableau obtained by removing $S$ from $T$. For $a\le b$ let $C_a^b$ be the
column word $b (b-1)\dotsm (a+1) a$.

\begin{proposition} \label{P:GL} Let $\emptyset\ne\ub\in\Row^*$,
$m=\min(\ub)$, $(P,Q)=(\bP(\ub),\bQ(\ub))$, $m=\min(P)$, and
$(P^\dagger,Q^\dagger)=(\bP(\ubd),\bQ(\ubd))$.
\begin{enumerate}
\item $\kh(\ub)$ is the maximum $\kh\in\Z_{\ge0}$ such that
$C=C_m^{m+\kh-1}$ is contained in the first column of $Q$.
\item The value $k$ in the definition of $\ubd$ equals the maximum $k$ such that
$C_m^{m+k-1}$ is contained in the first column of $P$.
\item We have
\begin{align}
	\label{E:Qlowerdagger}
		Q^\dagger = \bP(Q-C).
\end{align}
In particular $\shape(Q)/\shape(Q^\dagger) \in \cV_{\kh}$.
\item Let $(T,c)\in \Tab_{\shape(Q^\dagger)} \times \Col_{\kh}$ be the unique pair of Proposition \ref{P:Pieri} such that $\bP(Tc)=P$. Then $T=\bP(\Uh(\ub))$, $c=C$, and
$P^\dagger = U_{m+\kh}(T)$.
\end{enumerate}
\end{proposition}

\begin{example}\label{X:GLtableau} With $\ub$ as in Example \ref{X:PQ}
$(P^\dagger,Q^\dagger)$
are given in Figure \ref{F:tableaux}, computed as in Proposition \ref{P:GL}.
	\begin{figure}
		\begin{align*}
		P&=\begin{ytableau}
		*(yellow) 1&1&1&2 \\ *(yellow) 2&2&2\\ *(orange) 3&3&3\\*(orange)4&6\\*(orange)5
		\end{ytableau}&\qquad
		Q&=\begin{ytableau}
		*(yellow) 1&2&3&3\\ *(yellow)2&3&4\\ 4&4&5\\5&6\\6
		\end{ytableau} \\
		T\otimes C	&= \begin{ytableau} 1&1&2\\2&2&*(green)3
		\\3&3\\*(green)4&6\\*(green)5
		\end{ytableau} \otimes \begin{ytableau}*(yellow) 1 \\*(yellow)2\end{ytableau}
		& Q^\dagger=\bP(Q-C)&=
		\begin{ytableau}
		2&3&3\\3&4&4\\4&5\\5&6\\6
		\end{ytableau}\\
		P^\dagger &=\begin{ytableau} 1&1&2\\2&2&*(green)4\\3&3\\*(green)5&6\\*(green)6
		\end{ytableau}
		\end{align*}
		\caption{The computation of $Q^\dagger$ and $P^\dagger$ from $Q$ and $P$}
		\label{F:tableaux}
	\end{figure}
\end{example}

\subsection{Key tableaux}\label{appsec:key tableaux}
Say that a tableau is a \textit{key tableau} if it has the property that
any of its columns, viewed as a subset, contains any column to its right.

Consider a word with symbols in the set $\{1,2,\dotsc,n\}$.
The weight of a word $u$ is the sequence $\wt(u)=(m_1(u),m_2(u),\dotsc,m_n(u))\in \Z_{\ge0}^n$ where $m_j(u)$ is the multiplicity of the entry $j$ in $u$.

Given $\beta\in \Z_{\ge0}^n$ let $\beta^+$ be the unique partition in the orbit
$S_n \cdot\beta$ and let $w_\beta\in S_n$ be the shortest element such that
$w_\beta(\beta^+)=\beta$. The \textit{key tableau} $\mathbb{K}_\beta$ of weight $\beta$ is the unique tableau of shape $\beta^+$ and weight $\beta$.
Given a fixed partition $\lambda$ the orbit $S_n\cdot\lambda$ is a poset:
$\alpha\le \beta$ if $w_\alpha \le w_\beta$ in Bruhat order.

\subsection{Left and right keys; image of $\RSKn$} \label{appsec: keys}
Recall horizontal and vertical strips from \S \ref{sec: shapes}.
Say that a vertical (resp. horizontal) strip is $\lambda$-removable if it has the form $\lambda/\mu$ for some $\mu$.

Say that a vertical strip of size $r$ is \textit{grounded} if it contains a box in each of the
first $r$ rows. It is obvious that for a given $r$ a partition has at most one grounded removable vertical strip of size $r$. This strip exists if and only if
$\lambda$ has a column of size $r$.

Let $P$ be a tableau of shape $\lambda$. Let $L=\lambda_1$ be the number of columns of $\lambda$. For each column size $r = \lambda^t_j$ of $\lambda$, let $\mu\subset\lambda$ be the unique partition such that
$\lambda/\mu\in \cV_r$ is grounded. By Proposition \ref{P:Pieri}
there is a unique pair $(u,T)\in \Col_r \times \Tab_\mu$ (resp. $(T,u)\in \Tab_\mu\times \Col_r$ such that $\bP(uT)=P$ (resp. $\bP(Tu)=P$). Denote the word $u$ by $c_j^-(P)$ (resp. $c_j^+(P)$).

\begin{lemma} \label{L:key columns containment}
Considering column words as subsets, we have
$c_j^-(P) \supset c_{j+1}^-(P)$ and $c_j^+(P) \supset c_{j+1}^+(P)$
for all $j$.
\end{lemma}

The left key $K_-(P)$ (resp. right key $K_+(P)$) of $P$ is by definition the tableau of the same shape as $P$ whose $j$-th column is given by the column word $c_j^-(P)$
(resp. $c_j^+(P)$) for all $j$. By Lemma \ref{L:key columns containment} $K_-(P)$ and $K_+(P)$ are key tableaux and in particular semistandard.

\begin{example} For the tableaux $(P,Q)=(\bP(\ub),\bQ(\ub))$ from Example \ref{X:PQ}
the factorizations defining $c_J^+(P)$ and $c_j^-(Q)$ are given as follows. See Example \ref{X:right column factor} for the method of producing the 4-box right column factor from $P$ and Example \ref{X:left column factor} for how to obtain the 4-box left column factor from $Q$.
\begin{align*}
P&= \begin{ytableau}1&1&1&2\\2&2&2\\3&3&3\\4&6\\5\end{ytableau}
\equiv
\begin{ytableau}1&1&2\\2&2\\3&3\\5 \end{ytableau}\cdot
\begin{ytableau}1\\2\\3\\4\\6\end{ytableau}\equiv
\begin{ytableau} 1&1&2\\2&2\\3&3\\4\\5\end{ytableau} \cdot
\begin{ytableau} 1\\2\\3\\6 \end{ytableau}\equiv
\begin{ytableau} 1&1&2\\2&2\\3&3\\4&6\\5\end{ytableau} \cdot
\begin{ytableau} 1\\2\\3 \end{ytableau}\equiv
\begin{ytableau} 1&1&1\\2&2&2\\3&3&3\\4&6\\5\end{ytableau} \cdot
\begin{ytableau} 2\end{ytableau}\\
K_+(P)&=\begin{ytableau}
1&1&1&2\\2&2&2\\3&3&3\\4&6\\6
\end{ytableau}\qquad \wt(K_+(P)) = (3,4,3,1,0,2)
\end{align*}
\begin{align*}
Q&=\begin{ytableau} 1&2&3&3\\2&3&4\\4&4&5\\5&6\\6\end{ytableau}
\equiv \begin{ytableau} 1\\2\\4\\5\\6\end{ytableau} \cdot
\begin{ytableau} 2&3&3\\3&4\\4&5\\6\end{ytableau}
\equiv \begin{ytableau} 1\\2\\4\\6 \end{ytableau}\cdot
\begin{ytableau} 2&3&3\\3&4\\4&5\\5\\6\end{ytableau}
\equiv \begin{ytableau} 1\\2\\4 \end{ytableau}\cdot
\begin{ytableau} 2&3&3\\3&4\\4&5\\5&6\\6\end{ytableau}
\equiv \begin{ytableau} 2 \end{ytableau}\cdot
\begin{ytableau} 1&2&3\\3&3&4\\4&4&5\\5&6\\6\end{ytableau} \\
K_-(Q)&=
\begin{ytableau} 1&1&1&2\\2&2&2\\4&4&4\\5&6\\6\end{ytableau}
\qquad \wt(K_-(Q)) = (3,4,0,3,1,2).
\end{align*}
For $\alpha=\wt(K_+(P))=(3,4,3,1,0,2)$ we have $w_\alpha=s_1 s_5 s_4$
and for $\beta=\wt(K_-(Q))=(3,4,0,3,1,2)$ we have $w_\beta=s_1s_3s_5s_4s_5$.
We see that $w_\alpha \le w_\beta$.
\end{example}

\begin{theorem} \label{T:triangular key}
	Let $\ub\in\Row^*$. Then $\ub$ is flagged (in the image of a multisegment under $\iota$) if and only if
	$\wt(K_+(\bP(\ub))) \le \wt(K_-(\bQ(\ub)))$.
\end{theorem}

This theorem can be proved using the ideas of \cite{MR1324004}
combined with the combinatorics of the crystal graphs of Demazure modules
in highest weight modules for $U_q(\mathfrak{sl}_n)$.

\begin{remark} A similar condition works for enhanced multisegments;
the flagged condition for $\ub$ is changed to requiring that all values in $u^{(i)}$ be at most $i+1$, and that instead of Bruhat-comparing the weights $\alpha$ and $\beta$ of the
right key of $\bP(\ub)$ and the left key of $\bQ(\ub)$, one uses the weight
$(0,\beta_1,\beta_2,\dotsc)$ instead of $\beta$.
\end{remark}

\subsection{Theorem \ref{T:triangular key} in terms closer to inverted tableaux}
We wish to state the condition of Theorem \ref{T:triangular key} in terms of the
inverted tableau pair $(P_{\m},Q_{\m})$. The first transformation is to apply
$\flip$ so instead the condition is on a pair of antitableaux. It suffices to define
left and right key for an antitableau. By definition the left key and right key apply to a tableau. Since each Knuth class contains a unique tableau word and antitableau word,
it makes sense to talk about the left and right keys of an antitableau. This suffices to describe the required condition implicitly.

However, since the map between a tableau and its antitableau is nontrivial, we wish
to give a more direct way to compute the left and right key, directly from the antitableau.
For this it suffices to compute the column factors $c_j^\pm$ of an antitableau.

We will do this by example, using anti-analogues of the methods given in
Examples \ref{X:left column factor} and \ref{X:right column factor}.
	
\begin{example} \label{X:anti left column factor}
Consider the tableau $X$ of shape $\lambda=(4,3,3,2,1)$ from Example \ref{X:left column factor}. Suppose instead of $X$ we are given its antitableau. It is pictured in Figure \ref{F:anti}.

Let us compute the column $c_2^-(X^\anti)$, which has size $4$ since $\lambda^t_2=4$.
Let $\mu\subset\lambda$ be the unique partition such that $\lambda/\mu$ is the grounded $\lambda$-removable vertical strip of size $4$. We make a skew tableau $U$ of the anti-shape of $\lambda/\mu$ filled with the numbers $1$ through $4$ from top to bottom.

We now perform a sequence of successive internal row insertions on $X^\anti$ at the boxes of $U$ starting with the box labeled $1$ (cf. \cite{MR1075706}).
The bumping paths, which push entries downward, are in green.
\begin{figure}
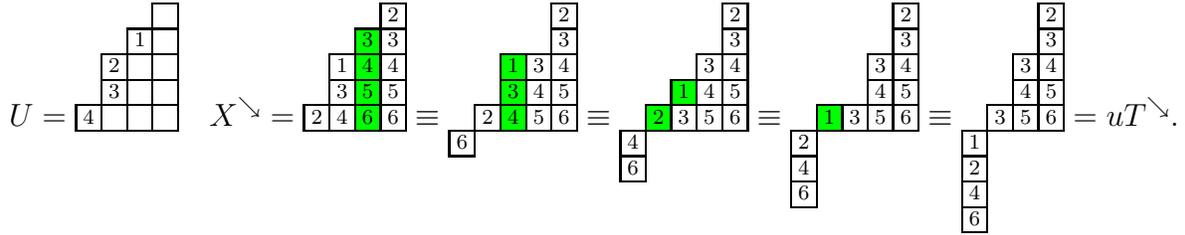

\begin{align*}
	U = \begin{ytableau}
		\none&\none&\none&{} \\
		\none&\none&1&{} \\
		\none&2&{}&{} \\
		\none&3&{}&{} \\
		4&{}&{}&{} \\ \none \\ \none\\ \none \\ \none \\
	\end{ytableau}\quad
X^\anti=\begin{ytableau}
	\none&\none&\none&2 \\
	\none&\none&*(green)3&3\\
	\none&1&*(green)4&4\\
	\none&3&*(green)5&5\\
	2&4&*(green)6&6\\
	\none \\ \none\\ \none\\ \none
\end{ytableau}	\equiv
\begin{ytableau}
	\none&\none&\none&\none&2 \\
	\none&\none&\none&\none&3\\
	\none&\none&*(green)1&3&4\\
	\none&\none&*(green)3&4&5\\
	\none&2&*(green)4&5&6\\
	6 \\ \none \\ \none \\ \none
\end{ytableau}\equiv
\begin{ytableau}
	\none&\none&\none&\none&2 \\
	\none&\none&\none&\none&3\\
	\none&\none&\none&3&4\\
	\none&\none&*(green)1&4&5\\
	\none&*(green)2&3&5&6\\
	4\\
	6 \\ \none\\ \none
\end{ytableau}\equiv
\begin{ytableau}
	\none&\none&\none&\none&2 \\
	\none&\none&\none&\none&3\\
	\none&\none&\none&3&4\\
	\none&\none&\none&4&5\\
	\none&*(green)1&3&5&6\\
	2\\
	4\\
	6 \\  \none
	\end{ytableau}\equiv
\begin{ytableau}
	\none&\none&\none&2 \\
	\none&\none&\none&3\\
	\none&\none&3&4\\
	\none&\none&4&5\\
	\none&3&5&6\\
	1\\
	2\\
	4\\
	6 \\
\end{ytableau} = u T^\anti.
\end{align*}
\caption{Computing the left key of an antitableau}
\label{F:anti}
\end{figure}
Example \ref{X:left column factor} produced a factorization $X\equiv u T$:
the current computation produces the factorization $X^\anti\equiv u T^\anti$.
\end{example}

\begin{example} \label{X:anti right column factor}
Example \ref{X:right column factor} computed a factorization
$Y \equiv T u$. We illustrate a factorization $Y^\anti \equiv T^\anti u$.
The antitableau $Y^\anti$ is pictured in Figure \ref{F:anti right}. Using the same tableau $U$ as in Example \ref{X:anti left column factor}
we apply successive internal column insertions on $Y^\anti$ at the boxes of $U$ in the same order as before. The bumping paths push to the right.
\begin{figure}
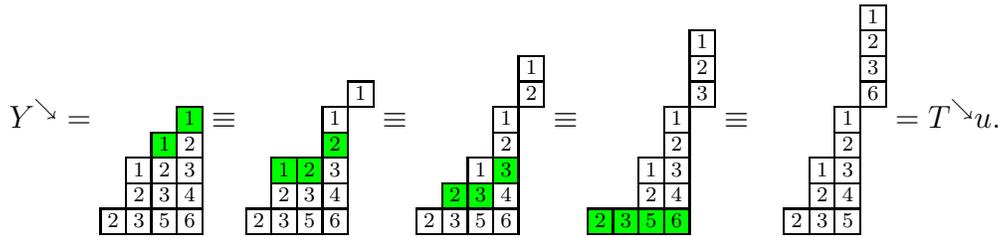

	\begin{align*}
			Y^\anti &= \begin{ytableau} \none \\ \none \\ \none \\ \none \\
				\none&\none&\none&*(green)1 \\
				\none&\none&*(green)1&2 \\
  			    \none&1&2&3 \\
  			  \none&2&3&4 \\
  		  	  2&3&5&6
	\end{ytableau}
\equiv
\begin{ytableau} \none \\ \none \\  \none \\ \none & \none & \none&\none &1 \\
	\none&\none&\none&1 \\
	\none&\none&\none&*(green)2 \\
	\none&*(green)1&*(green)2&3 \\
	\none&2&3&4 \\
	2&3&5&6
\end{ytableau}
\equiv
\begin{ytableau} \none \\  \none \\ \none & \none & \none&\none &1 \\
	\none&\none&\none&\none&2 \\
		\none&\none&\none&1 \\
	\none&\none&\none&2 \\
	\none&\none&1&*(green)3 \\
	\none&*(green)2&*(green)3&4 \\
	2&3&5&6
\end{ytableau}
\equiv
\begin{ytableau}  \none \\ \none & \none & \none&\none &1 \\
	\none&\none&\none&\none&2 \\
		\none&\none&\none&\none&3 \\
	\none&\none&\none&1 \\
	\none&\none&\none&2 \\
	\none&\none&1&3 \\
	\none&\none&2&4 \\
	*(green)2&*(green)3&*(green)5&*(green)6
\end{ytableau}
\equiv
\begin{ytableau} \none & \none & \none&\none &1 \\
	\none&\none&\none&\none&2 \\
	\none&\none&\none&\none&3 \\
		\none&\none&\none&\none&6 \\
	\none&\none&\none&1 \\
	\none&\none&\none&2 \\
	\none&\none&1&3 \\
	\none&\none&2&4 \\
	\none&2&3&5
\end{ytableau} = T^\anti u.
\end{align*}
\caption{Computing the right key of an antitableau}
\label{F:anti right}
\end{figure}
\end{example}


\def\cprime{$'$}

\end{document}